\documentclass{article}
\usepackage[utf8]{inputenc}

\usepackage[dvipsnames]{xcolor}
\usepackage{latexsym,amssymb}
\usepackage{amsmath,amsthm,amsxtra,amsbsy,accents}
\usepackage{mathtools}
\mathtoolsset{showonlyrefs}
\usepackage[scr=boondoxo]{mathalpha}
\usepackage{bm}
\usepackage{hyperref}
\usepackage{subfig}
\usepackage[a4paper,left=2.5cm,top=3cm,right=3cm,bottom=4cm]{geometry}
\usepackage{a4wide}
\usepackage{authblk}
\usepackage{ulem}  

\renewcommand*\arraystretch{1.6}

\usepackage{enumitem}

\usepackage{upgreek}

\usepackage[only,llbracket,rrbracket]{stmaryrd}

\newcommand{\mala}[1]{\mbox{\large $#1$}}

\DeclareMathAlphabet{\mathup}{OT1}{\familydefault}{m}{n}

\newcommand{\dd}{\mathop{}\mathrm{d}} %von Gianna

\newcommand{\lam}{\lambda}
\newcommand{\eps}{\varepsilon}

\renewcommand{\phi}{\varphi}

\newcommand{\tils}{\tilde{\sigma}}
\newcommand{\tilu}{\tilde{u}}
\newcommand{\tilv}{\tilde{v}}

\newcommand{\tilp}{\tilde{p}}
\newcommand{\tilq}{\tilde{q}}
\newcommand{\eni}{\textit{(i)}}
\newcommand{\enii}{\textit{(ii)}}
\newcommand{\eniii}{\textit{(iii)}}

\def\Xint#1{\mathchoice
	{\XXint\displaystyle\textstyle{#1}}%
	{\XXint\textstyle\scriptstyle{#1}}%
	{\XXint\scriptstyle\scriptscriptstyle{#1}}%
	{\XXint\scriptscriptstyle\scriptscriptstyle{#1}}%
	\!\int}
\def\XXint#1#2#3{{\setbox0=\hbox{$#1{#2#3}{\int}$}
		\vcenter{\hbox{$#2#3$}}\kern-.5\wd0}}

\def\dashint{\Xint-}
\newcommand{\mint}{\dashint}

\makeatletter
\newcommand{\customlabel}[2]{%
   \protected@write \@auxout {}{\string \newlabel {#1}{{#2}{\thepage}{#2}{#1}{}} }%
   \hypertarget{#1}{}
}
\makeatother

\numberwithin{figure}{section}

\def\AAA{{\mathcal A}}

\def\rmj{{\mathrm j}}

 \def\N{{\mathbb N}}

 \def\Z{{\mathbb Z}}

\newcommand{\R}{\mathbb{R}}

\makeatletter
\newcommand{\bnabla}{{\mathpalette\b@nabla\relax}}
\newcommand\b@nabla[2]{%
        \setbox\z@=\hbox{$\m@th#1\bigtriangledown$}%
        \ht\z@.7\ht\z@
        \raise\dp\z@\box\z@
}
\makeatother

\newtheorem{theorem}{Theorem}[section]

\newtheorem{lemma}[theorem]{Lemma}
\newtheorem{cor}[theorem]{Corollary}

\newtheorem{prop}[theorem]{Proposition}
\newtheorem{rem}[theorem]{Remark}
\newtheorem{rem*}{Remark}

\numberwithin{equation}{section}

\makeatletter
\newcommand{\oset}[3][0ex]{%
  \mathrel{\mathop{#3}\limits^{
    \vbox to#1{\kern-2\ex@
    \hbox{$\scriptstyle#2$}\vss}}}}
\makeatother

\usepackage{trimclip}
\makeatletter
\NewDocumentCommand{\rmjmath}{}{\mathbin{\mathpalette\eplus@\relax\mspace{1mu}}}
\newcommand{\eplus@}[2]{\clipbox{-.5 -.5 0 {.35\height}}{$\m@th#1\rmj$}}
\makeatother

\usepackage{tikz}
\usetikzlibrary{arrows.meta}
\usetikzlibrary{calc,intersections,through}

\definecolor{bluegray}{rgb}{0.4, 0.6, 0.8}

\title{$\Gamma$-Convergence of Higher-Order Phase Transition Models}

\author[1]{Denis Brazke}
\author[2]{Gianna G\"{o}tzmann}
\author[3]{Hans Kn\"{u}pfer}
\affil[1]{{\small Okinawa Institute of Science and Technology, Analysis and Partial Differential Equations Unit, 1919-1 Tancha Onna-son, 904-0495 Japan. {denis.brazke@oist.jp}}}
\affil[2]{{\small University of Augsburg, Institute of Mathematics, Universit\"{a}tsstra\ss e 12a, 86159 Augsburg, Germany. {gianna.goetzmann@uni-a.de}}}
\affil[3]{{\small Heidelberg University, Institute for Mathematics, Im Neuenheimer Feld 205, 69120  Heidelberg, Germany. {knuepfer@uni-heidelberg.de}}}

\date{\today}

\begin{document}

\maketitle

\begin{abstract}
We investigate the asymptotic behavior as $\eps \to 0$ of singularly perturbed phase transition models of order $n \geq 2$, given by
\begin{align}
    G_\eps^{\lam,n}[u] \coloneq
    \int_I \frac 1\eps W(u) -\lam\eps^{2n-3} (u^{(n-1)})^2 + \eps^{2n-1} (u^{(n)})^2\dd x, \quad u \in W^{n,2}(I),
\end{align}
 where $\lam >0$ is fixed, $I \subset \R$ is an open bounded interval, and $W \in C^0(\R)$ is a suitable double-well potential. We find that there exists a positive critical parameter depending on $W$ and $n$, such that the $\Gamma$-limit of $G_\eps^{\lam,n}$ with respect to the $L^1$-topology is given by a sharp interface functional in the subcritical regime. The cornerstone for the corresponding compactness property is a novel nonlinear interpolation inequality involving higher-order derivatives, which is based on Gagliardo-Nirenberg type inequalities.
\end{abstract}

\tableofcontents

\newpage

\section{Introduction}

We investigate the asymptotic behavior of a one-dimensional higher-order Ginzburg-Landau type model, which is given by the family of free-energy functionals
\begin{align} \label{GGn}
    G_\eps^{\lam,n}[u] \coloneq
     \int_I \frac 1\eps W(u) -\lam\eps^{2n-3} (u^{(n-1)})^2 + \eps^{2n-1} (u^{(n)})^2\dd x, \quad u \in W^{n,2}(I), 
\end{align}
where $I \subset \R$ is an open bounded interval, $\lam >0 $ is a constant, $n \in \N_{\geq2}$ is the order of the highest derivative, and $W \in C^0(\R)$ is a suitable double-well potential. Functionals of this form are used to model the behavior of complex materials undergoing phase separation processes (see e.g. \cite{ColMiz}). Using the framework of $\Gamma$-convergence, we derive an asymptotic effective model for the family of functionals given in  \eqref{GGn} as $\eps \to 0$. More precisely, we show that for every order $n$ there exists a critical parameter $\lam_n > 0$ such that the $\Gamma$-limit of $G^{\lam,n}_\eps$ with respect to the $L^1$-topology in the subcritical regime $\lam \in (0,\lam_n)$ is given by a constant multiple of the perimeter functional.

The energy $G_\eps^{\lam,n}$ is a generalization of several known phase separation models. 
For instance, the family of functionals $G_\eps^{0,1}$ represents a classical first-order phase transition model introduced by Cahn and Hilliard \cite{CahHil}. It describes the free energy of a system where phase separation processes between two coexisting isotropic liquids occur. From the theory of thermodynamics it is known that, for many materials, the interface between these phases becomes thinner with decreasing temperature (see e.g. \cite{RowSwi}). This effect motivates the investigation of the functional's asymptotic behavior as the parameter $\eps$, describing the non-dimensionalized thickness of the transition layer, tends to zero. Based on the work of Modica and Mortola \cite{ModMor}, Modica \cite{Mod} and Sternberg \cite{Ster} demonstrated a certain compactness property for sequences with equi-bounded energy and $\Gamma$-convergence of $G_\eps^{0,1}$ as $\eps \to 0$ towards a sharp interface functional, provided that the double-well potential $W$ satisfies a linear coercivity condition. Corresponding results were later shown for $G_\eps^{0,2}$ by Fonseca and Mantegazza \cite{FonMan}, and extended to higher orders by Brusca et al. \cite{BruDonSol} who presented an examination of $G_\eps^{0,n}$ for arbitrary $n \in \N$. Specifically, they derived a suitable general compactness property and showed that the sequence $G_\eps^{0,n}$ $\Gamma$-converges to a sharp interface functional.

Coleman et al. \cite{ColMiz} introduced the second-order model $G_\eps^{\lam,2}$. It describes phase separation processes in nonlinear materials that exhibit periodic layering phenomena, for instance, in concentrated soap solutions or metallic alloys. They argue that if the negative contribution exceeds a certain threshold, the phases become unstable and minimizers of this energy may develop a periodic structure. This was later rigorously confirmed by Mizel et al. \cite{MizPelTro} using variational methods. Therein, they proved that for the standard double-well potential $W(t)=(t-1)^2(t+1)^2$ and sufficiently large $\lam>0$, the minimal energy tends to $-\infty$ with increasing oscillation as $\eps \to 0$. However, Cicalese et al. \cite{CicSpaZep} and Chermisi et al. \cite{CheDalFonLeo} studied the asymptotic behavior of $G_\eps^{\lam,2}$ through $\Gamma$-convergence for small $\lam>0$. They found that there exists a critical constant $\bar \lam > 0$ depending on $W$ such that for $\lam \in (0,\bar \lam)$, the $\Gamma$-limit is again given by a sharp interface functional if $W$ satisfies a quadratic coercivity condition. While Chermisi et al. \cite{CheDalFonLeo} focused on considering the model in higher dimensions, Cicalese et al. \cite{CicSpaZep} dealt with the one-dimensional problem but provided an upper bound for $\bar \lam$ such that no oscillations of the minimizers occur and a lower bound such that oscillations certainly arise. In addition, Hilhorst et al. \cite{HilPelSch} provided a consistent $\Gamma$-convergence result for $\lam<0$. 

In the case $\lam > 0$, the energy functional $G_\eps^{\lam,2}$ incorporates a strictly concave term, making it difficult to find uniform lower bounds. To overcome this issue, Cicalese et al. \cite{CicSpaZep} and Chermisi et al. \cite{CheDalFonLeo} found nonlinear interpolation inequalities, which allowed for a lower bound of the energy in terms of $G_\eps^{0,2}$ and $G_\eps^{0,1}$ respectively, provided the loss of convexity is not too strong (meaning $\lam > 0$ is sufficiently small). Furthermore, nonlinear interpolation inequalities are also used in the theory of nonlocal phase separation models, where control in terms of positive local phase separation models is achieved (see e.g. \cite{FonHayLeoZwi, GinPesZwi}, see also \cite{Sol}).

In this work, we aim to find a matching theory for the higher-order functionals $G_\eps^{\lam,n}$ given by \eqref{GGn}. The starting point of our analysis are the results from \cite{CicSpaZep}, where the related second-order functional was investigated, and impose the same quadratic coercivity condition on the double-well potential (see {\rm{(\hyperref[it-W3]{W3})}}). The latter is needed in order to derive an adapted nonlinear interpolation inequality (see Section~\ref{Subsec:NonLin}) which in turn is used to obtain $L^1$-compactness for sequences with equi-bounded energy (see Section~\ref{Subsec:LowBound}). Finally, we show $\Gamma$-convergence towards a sharp interface functional as $\eps \to 0$ (see Section~\ref{Sec:Gamma}).

This article is organized as follows. In Section \ref{Sec:Prelim}, we present the mathematical setup, preliminaries, and state our main results. In Section \ref{Sec:Compact}, we focus on the consideration and derivation of suitable interpolation inequalities and prove compactness for our family of energy functionals. In Section \ref{Sec:Gamma}, we prove its $\Gamma$-convergence. Some auxiliary interpolation inequalities are moved to the appendix.

\section{Preliminaries and Main Results} \label{Sec:Prelim}

 In the following, we always assume $n \in \N_{\geq2}$ and $\lam \in \R$, $\lam>0$. We restrict ourselves to the one-dimensional case where $I \subset \R$ denotes an open bounded interval of length $|I|$. Given $1 \leq p < \infty$ and $m \in \N$, we denote as usual the Lebesgue spaces by $L^p(I)$ and the Sobolev spaces by $W^{m,p}(I)$ with corresponding norms
    \begin{align}
        \|f\|^p_{L^p(I)} \coloneqq \int_I |f(x)|^p \dd x, && \|f\|_{W^{m,p}(I)} \coloneqq \sum_{j = 0}^m \|f^{(j)}\|_{L^p(I)},
    \end{align}
 where $f^{(j)}$ denotes the $j$-th weak derivative of $f$. We denote by $BV(I,\{\pm1\})$ the space of functions with bounded variation only taking the values $\pm 1$ almost everywhere, i.e. $u \in L^1(I,\{\pm 1\})$ and
    \begin{align}
        V[u] \coloneqq \sup \left\{ \int_I u(x) \, \varphi'(x) \dd x : \varphi \in C_c^\infty(I), \ \|\varphi\|_{L^\infty(I)} \leq 1 \right\} < \infty.
    \end{align}
It is well known that $u \in BV(I,\{\pm 1\})$ has a piecewise constant representative $\bar u$ with finitely many points of discontinuities. We note that $V[u] = 2 \# S(\bar u)$, where $S(\bar u) \subset I$ denotes the set of discontinuity points of $\bar u$. For more details, see \cite{EvaGar}. In the following, we implicitly identify $u$ with $\bar u$. In this article, we use $C$ as a generic constant which might differ from line to line.

We fix the double-well potential $W \colon \R \to \R$ and impose the following assumptions.
    \begin{enumerate}
        \item[(W1)] \label{it-W1} $W$ is continuous and $W \geq 0$.
        \item[(W2)] \label{it-W2} $W(t) = 0$ if and only if $t = \pm 1$.
        \item[(W3)] \label{it-W3} There exists $L > 0$ such that $ W(t) \geq L(t \mp 1)^2$ for all $\pm t > 0$. \hfill (Coercivity)
    \end{enumerate} \newcommand{\hyp}[1]{{\rm{(\hyperref[it-W#1]{W#1})}}}
These assumptions on the double-well potential match the assumptions already imposed e.g. in \cite{CicSpaZep, CheDalFonLeo}. Lastly, given $\eps > 0$, we define $G_\eps^{\lam, n} \colon L^1(I) \to \R \cup \{+\infty\}$ via
    \begin{align}
        G_\eps^{\lam, n}[u] \coloneqq \left\{ \begin{array}{ll}
             \displaystyle \int_I \frac 1\eps W(u) -\lam \eps^{2n-3} (u^{(n-1)})^2 + \eps^{2n-1} (u^{(n)})^2 \dd x    & \quad \text{for all } u \in W^{n,2}(I), \\[10pt]
             +\infty                                                                                                  & \quad \text{for all } u \in L^1(I) \setminus W^{n,2}(I). 
        \end{array} \right.
    \end{align}

Our first main theorem establishes a novel nonlinear interpolation inequality of arbitrary order of differentiation, based on the Gagliardo-Nirenberg inequality \cite{Gag,Nir59} and the techniques developed in \cite[Section~3]{CicSpaZep}.

\begin{theorem}[Nonlinear interpolation for higher orders] \label{Theo:NonLinInt_n}
Let $n \in \N_{\geq2}$ and assume $W$ satisfies \hyp1 -- \hyp3. Then there exists $\lam_n = \lam_n(L)>0$ such that
    \begin{align} \label{NonLinInt_n}
        \lam_n \int_I (u^{(n-1)})^2\dd x \leq \frac 1{|I|^{2n-2}} \int_I W(u)\dd x + |I|^2 \int_I (u^{(n)})^2\dd x
    \end{align}
for every open bounded interval $I \subset \R$, and every $u \in W^{n,2}(I)$.
\end{theorem}

The proof is given at the end of Section \ref{Subsec:NonLin}. The main difficulties lie in the incorporation of the nonlinear double-well potential $W$, and in obtaining suitable scaling factors. Replacing the $L^2$-norm by the nonlinear term in the interpolation requires the use of the coercivity condition \hyp3 on intervals where $u$ is either strictly positive or strictly negative and therefore an examination of the roots of $u$.

\begin{rem}
    We note that a subquadratic coercivity condition, e.g. linear coercivity, instead of \hyp3 would not suffice to show \eqref{NonLinInt_n}. This is because the double-well potential does not provide sufficient control over fine-scale oscillations in the subquadratic case. A detailed elaboration of this reasoning for the case $n=2$ can be found in \cite[Section 3.1]{CicSpaZep}.
\end{rem}

We fix $\lam_n > 0$ to be the optimal constant established in Theorem \ref{Theo:NonLinInt_n}. We call $\lam$ subcritical if $0 < \lam < \lam_n$ and supercritical if $\lam > \lam_n$. Theorem \ref{Theo:NonLinInt_n} allows us to control the concave term in $G_\eps^{\lam,n}$ by the remaining non-negative terms in the subcritical regime. In Section~\ref{Subsec:LowBound}, we use this control to estimate $G^{\lam,n}_\eps$ from below by $G^{0,n}_\eps$ up to a constant. Brusca et al. \cite{BruDonSol} already found that sequences with equi-bounded energy $G^{0,n}_\eps$ are precompact with respect to convergence in measure. Thus, we obtain compactness in $L^1$ of the energy functional $G_\eps^{\lam,n}$ by combining the result from \cite{BruDonSol} and Vitali's Convergence Theorem. We note that in the following $\eps \to 0$ means that the statement holds for every sequence $\eps_m \to 0$ as $m \to \infty$.

\begin{theorem}[Compactness] \label{Theo:comp_n}
Let $n \in \N_{\geq2}$, $\lam \in (0,\lam_n)$ and assume $W$ satisfies \hyp1 -- \hyp3. Let $I \subset \R$ be an open bounded interval. Let $u_\eps \in W^{n,2}(I)$ such that $\limsup_{\eps\to 0} G_\eps^{\lam,n}[u_\eps] < \infty$. Then there exists $u \in BV(I,\{\pm1\})$ and a subsequence (not relabeled) such that $u_\eps \to u$ in $L^1(I)$ as $\eps \to 0$.
\end{theorem}

Theorem \ref{Theo:comp_n} is proved at the end of Section \ref{Sec:Compact}. Our last main theorem deals with the asymptotic behavior of the family of energies $G_\eps^{\lam,n}$ as $\eps \to 0$. For this purpose, we introduce the sharp interface functional
    \begin{align}  \label{Gammalimes}
        G^{\lam,n}[u] \coloneqq \left\{\begin{array}{ll}
            C_W^{\lam,n} \#S(u),    & \quad \text{for all } u \in BV(I,\{\pm1\}), \\
            +\infty,                & \quad \text{for all } u \in L^1(I) \setminus BV(I,\{\pm1\}),
        \end{array}\right.
    \end{align}
where $C_W^{\lam,n}$ is given by the optimal profile problem
    \begin{align}
        C_W^{\lam,n}    & \coloneqq \inf\left\{\int_\R W(f) - \lam(f^{(n-1)})^2 + (f^{(n)})^2\dd x : f \in \AAA^n(\R) \right\}, \\
        \AAA^n(\R)      & \coloneqq \left\{ f \in W^{n,2}_{\mathrm{loc}}(\R) \colon f(x) = 1 \text{ for } x > T \text{ and } f(x) = -1 \text{ for } x < -T  \text{ for some } T>0 \right\}.
    \end{align}
While the explicit value of $C_W^{\lam,n}$ is not known to us, we note that $C_W^{\lam,n} > 0$ for $\lam \in(0,\lam_n)$, see Proposition \ref{constwohl}. We have the following $\Gamma$-convergence result.

\begin{theorem}[$\Gamma$-convergence] \label{Gamma_conv_n}
Let $n \in \N_{\geq 2}$, $\lam \in (0,\lam_n)$ and assume $W$ satisfies \hyp1 -- \hyp3. Let $I \subset \R$ be an open bounded interval. Then $G_\eps^{\lam, n} \stackrel{\Gamma}{\longrightarrow} G^{\lam, n}$ as $\eps \to 0$ in the $L^1$-topology, that is 
    \begin{itemize}
        \item for every $u \in L^1(I)$ and every $u_\eps \in L^1(I)$ such that $u_\eps \to u$, we have
            \begin{align} \label{defliminf} 
                G^{\lam, n}[u] \leq \liminf_{\eps \to 0} G_\eps^{\lam, n}[u_\eps].  \tag{\textit{Liminf inequality}}
            \end{align}%
        \item for every $u \in L^1(I)$ there exists $u_\eps \in L^1(I)$ such that $u_\eps \to u$ and 
            \begin{align} \label{deflimsup}
                \limsup_{\eps \to 0} G_\eps^{\lam, n}[u_\eps] \leq G^{\lam, n}[u].  \tag{\textit{Limsup inequality}}
            \end{align}
    \end{itemize}
\end{theorem}
\noeqref{deflimsup} \noeqref{defliminf}

The Theorem follows from the combination of Proposition~\ref{prop:liminf_inequality} and Proposition~\ref{prop:limsup_inequality}, which are proved in Section \ref{Sec:Gamma}. It reveals that, in the subcritical regime, the diffuse interface energy $G_\eps^{\lam, n}$ converges to the sharp interface model $G^{\lam, n}$.

The Limsup inequality is obtained by a well established approach from the theory of singular perturbation models (see e.g. \cite{CicSpaZep,CheDalFonLeo,BruDonSol}). The primary challenge is the proof of the Liminf inequality. In order to obtain the desired estimate, for every sequence $u_\eps \in W^{n,2}(I)$ under consideration, we construct an energetically favored competitor sequence (see Section~\ref{Subsec:FavComp}). It is necessary to make further considerations due to the higher-order derivatives, in particular we must ensure that the favored competitor belongs to $W^{n,2}(I)$. This is achieved by a suitable coupling of the two phases, whose energy is controlled using a suitable version of the Gagliardo-Nirenberg inequality (see Theorem~\ref{GagNir}).

\begin{rem}
As usual in phase transition models, one can impose a mass constraint by considering the energy functional by changing the class of admissible functions from $W^{n,2}(I)$ to
    \begin{align}
        \tilde{\AAA^n}(I) \coloneqq \left\{u \in W^{n,2}(I) : \int_I u(x) \dd x = m \right\}
    \end{align}
for some fixed $m \in (-|I|,|I|)$. However, the changes in the analysis are minor, and therefore we refrain from imposing it.
\end{rem}

\begin{rem}
It might be worthwhile to rigorously generalize the theory to higher dimensions by considering open and bounded domains $\Omega \subset \R^d$ instead of intervals $I \subset \R$, which is usually done by developing a higher-dimensional variant of the nonlinear interpolation inequality, using Fubini's Theorem and blow-up arguments (see \cite{CheDalFonLeo}, see also \cite{brai}). In addition, as in the case $n = 2$, the explicit dependence of the critical parameter $\lam_n$ on $W$ and $n$ is unknown to us.
\end{rem}

\section{Compactness} \label{Sec:Compact}

In this section, we prove the compactness property given by Theorem~\ref{Theo:comp_n}, which is needed to show our $\Gamma$-convergence result. To this end, in Section~\ref{Subsec:NonLin}, we establish a novel Gagliardo-Nirenberg type inequality that contains the nonlinear double-well potential term. In Section~\ref{Subsec:LowBound}, using this interpolation inequality, we find a lower bound for $G_\eps^{\lam,n}$ up to a constant, namely the energy $G_\eps^{0,n}$ for which a compactness result has already been shown by Brusca et al. \cite{BruDonSol}. In the last step, we transfer the compactness statement into the $L^1$-topology using Vitali's Convergence theorem.

\subsection{Nonlinear interpolation} \label{Subsec:NonLin}

In order to obtain a lower bound for $G_\eps^{\lam,n}$, we exploit the structure of the energy functional and estimate the second highest-order term by the highest-order term and the potential term. This recalls the Gagliardo-Nirenberg inequality from Theorem~\ref{GagNir}. In order to adapt it  to the present setting, we first consider another auxiliary interpolation, which is a variant of Lemma~\ref{IntLem} for the special case $p=q=r=2$ and can also be found in \cite[Lemma 3.3]{CheDalFonLeo}. The main difference lies in the more general scaling factors, which are relevant to the proof of the nonlinear interpolation in Theorem~\ref{NonLinInt_n}.

\begin{lemma} \label{Lemma:NirIneq_n}
Let $n \in \N_{\geq2}$. There exists $c = c(n) \in (0,1)$ such that
    \begin{align} \label{eq:NirIneq_n}
        c \int_I (u^{(n-1)})^2\dd x \leq \frac 1{\sigma^{2n-2}} \int_I u^2\dd x + \sigma^2 \int_I (u^{(n)})^2\dd x
    \end{align}
for every open bounded interval $I \subset \R$, every $0 < \sigma \leq |I|$ and every $u \in W^{n,2}(I)$.
\end{lemma}

\begin{proof}
    We prove the inequality by induction on $n$ starting with $n = 2$. The case $\sigma = |I|$ is covered by Lemma~\ref{IntLem} with the choice $p = q = r = 2$. Now, for $0 < \sigma < |I|$, we set $m \coloneqq \max\{N \in \N: N < \frac{|I|}{\sigma}\}$. We subdivide the interval $I$ into $m$ disjoint open intervals $I_1\ldots,I_m$ each of length $\frac{|I|}{m}=:\tils$ and observe that $m\sigma < |I| < 2m \sigma$ which is equivalent to $\sigma < \tils < 2 \sigma$. Using Lemma~\ref{IntLem}, we have
    \begin{align} \label{basecase}
    \begin{split}
        \int_I (u')^2\dd x & = \sum_{i=1}^m \int_{I_i} (u')^2\dd x
        \leq C_1 \sum_{i=1}^m \left(\frac {1}{\tils^2} \int_{I_i} u^2\dd x + \tils^2 \int_{I_i} (u'')^2\dd x \right)\\
        & = C_1 \left(\frac {1}{\tils^2} \int_{I} u^2\dd x + \tils^2 \int_{I} (u'')^2\dd x \right) 
        \leq 4C_1 \left(\frac {1}{\sigma^2} \int_{I} u^2\dd x + \sigma^2 \int_{I} (u'')^2\dd x \right).
    \end{split}
    \end{align}
    Now we assume that \eqref{eq:NirIneq_n} holds for a fixed $n \in \N_{\geq2}$, i.e. that there exists $C_2>1$ such that
    \begin{align} \label{Nir_IA}
    \int_I (u^{(n-1)})^2\dd x \leq C_2 \left(\frac 1{\sigma^{2(n-1)}} \int_I u^2\dd x + \sigma^2 \int_I (u^{(n)})^2\dd x\right)
    \end{align}
    for every $0 < \sigma \leq |I|$. Let $0< \sigma_1 \leq |I|$ be arbitrarily chosen and fixed. Using \eqref{basecase} with $u^{(n-1)}$ instead of $u$, there exists $C_3>1$ such that
    \begin{align}
        \int_I (u^{(n)})^2\dd x & \leq C_3 \left(\frac 1{\sigma_1^2} \int_I (u^{(n-1)})^2\dd x + \sigma_1^2 \int_I (u^{(n+1)})^2\dd x\right)\\
        & \hspace{-0.138cm} \overset{\eqref{Nir_IA}}{\leq} C_3 \left( \frac {C_2}{\sigma_1^2} \left(\frac {1}{\sigma_2^{2(n-1)}} \int_I u^2\dd x + \sigma_2^2 \int_I (u^{(n)})^2\dd x\right) + \sigma_1^2 \int_I (u^{(n+1)})^2\dd x\right),
    \end{align}
    where $0 < \sigma_2 \coloneq  \frac{\sigma_1}{\sqrt{2C_2C_3}} < \sigma_1 \leq |I|$. We hence obtain
    \begin{align}
        \left(\frac 1{C_3} - \sigma_2^2 \frac{C_2}{\sigma_1^2} \right) \int_I (u^{(n)})^2\dd x 
        & \leq \frac {C_2}{\sigma_1^2\sigma_2^{2(n-1)}} \int_I u^2\dd x + \sigma_1^2 \int_I (u^{(n+1)})^2\dd x.
    \end{align}
    Substituting the definition of $\sigma_2$ and estimating  $2^{n-1}C_2^nC_3^{n-1}>1$ finally yields \eqref{eq:NirIneq_n}.
\end{proof}

 We are now in a position to prove Theorem \ref{Theo:NonLinInt_n}. We make use of the interpolation inequality from Lemma~\ref{Lemma:NirIneq_n} and take advantage of the quadratic coercivity condition \hyp3 on the double-well potential. The latter holds on intervals where $u$ is either strictly positive or strictly negative and therefore a careful examination of the roots of $u$ is necessary. The foundation of the proof is the corresponding proof for the second-order version by Cicalese et al. \cite{CicSpaZep}.

\begin{proof}[Proof of Theorem \ref{Theo:NonLinInt_n}.]
    By translation, we may assume $I=(0,l)$ for some $l>0$. We define 
        \begin{align}
            m\coloneq  \mint_0^l u^{(n-1)}\dd x = \frac{u^{(n-2)}(l)-u^{(n-2)}(0)}{l}.
        \end{align}
    By the Mean Value Theorem, there exists $y_0 \in I$ such that $u^{(n-1)}(y_0)=m$. From the Fundamental Theorem of Calculus we obtain for all $y \in I$
        \begin{align} \label{u-m_Ska_n}
        \begin{split}
            |u^{(n-1)}(y)-m| = |u^{(n-1)}(y)-u^{(n-1)}(y_0)| = \left|\int_y^{y_0} u^{(n)}\dd x\right| \leq \int_0^l |u^{(n)}|\dd x,
        \end{split}
        \end{align}
    and thus, using Young's and Jensen's inequality,
        \begin{align}
            \int_0^l (u^{(n-1)})^2 \dd y & 
            \leq \int_0^l (|u^{(n-1)}-m|+|m|)^2 \dd y
            \overset{\eqref{u-m_Ska_n}}{\leq} \int_0^l 2 l \int_0^l (u^{(n)})^2\dd x + 2 m^2 
            \dd y \\
            &= 2 l^2 \int_0^l (u^{(n)})^2\dd x + 2 l m^2.
        \end{align}
    Therefore it is enough to show that there exists $C>0$ such that
        \begin{align} \label{zz_Ska_n}
            l m^2 \leq C \left(\frac 1{l^{2n-2}} \int_0^l W(u)\dd x + l^2 \int_0^l(u^{(n)})^2\dd x\right).
        \end{align}
    Without loss of generality, we assume $m^2 > 4 l\int_0^l (u^{(n)})^2\dd x$, otherwise the claim follows immediately from $W \geq 0$. Using Jensen's inequality, we then have $2 \| u^{(n)}\|_{L^1(I)} < |m|$ and hence
        \begin{align} \label{m<u'_Ska_n}
            2|m|-2|u^{(n-1)}(y)| \leq 2 |m-u^{(n-1)}(y)| \overset{\eqref{u-m_Ska_n}}{\leq} 2 \int_0^l |u^{(n)}|\dd x < |m|.
        \end{align}
   Thus, we obtain $0 \leq \frac 12 |m| < |u^{(n-1)}(y)|$ for all $y \in I$. Consequently, applying Rolle's Theorem inductively yields that $u$ has at most $n-1$ roots in $I$. We denote by $(\alpha,\beta) \subset I$ some subinterval of maximal length without roots of $u$, where $\beta - \alpha \geq \frac {l}n$. Without loss of generality, we assume $u>0$ on $(\alpha,\beta)$. We now apply Lemma~\ref{Lemma:NirIneq_n} with $\sigma=\frac ln$ and the coercivity condition \hyp3 to $u-1$, and obtain
        \setlength\jot{0.5cm}
        \begin{align} \label{UseCoerc}
        \begin{split}
            c \int_\alpha^\beta (u^{(n-1)})^2\dd x 
            & \overset{\eqref{eq:NirIneq_n}}{\leq} \frac {n^{2n-2}}{l^{2n-2}} \int_\alpha^\beta (u-1)^2\dd x + \frac {l^2}{n^2} \int_\alpha^\beta (u^{(n)})^2\dd x\\
            & \overset{\hyp3}{\leq} n^{2n-2} \max\left\{\frac 1L,\frac 1{n^{2n}}\right\} \left(\frac {1}{l^{2n-2}}  \int_0^l  W(u)\dd x + l^2  \int_0^l(u^{(n)})^2\dd x\right).
        \end{split}
        \end{align}
    Moreover, using \eqref{m<u'_Ska_n}, we obtain 
        \begin{align} \label{finalstep}
            \frac {l}{n} m^2 \leq \int_\alpha^\beta m^2\dd x \leq 4 \int_\alpha^\beta (u^{(n-1)})^2\dd x
        \end{align}
    and hence, combining \eqref{UseCoerc} and \eqref{finalstep}, the claim \eqref{zz_Ska_n}.
\end{proof}

From now on, the constant $\lam_n = \lam_n(L)>0$ is fixed to be the optimal constant in Theorem~\ref{Theo:NonLinInt_n}. As alluded to in Section \ref{Sec:Prelim}, it separates the sub- and supercritical regime for the parameter $\lam$.

Analogously to \cite[Remark~3.2]{CicSpaZep}, by subdividing $\R$ into intervals of the form $(k,k+1)$ with $k \in \Z$ and applying \eqref{NonLinInt_n} to each of these intervals, we obtain the following version of Theorem~\ref{Theo:NonLinInt_n} on the real line.

\begin{cor}[Nonlinear interpolation on the real line] \label{corNonLinInt_R}
Let $n \in \N_{\geq2}$ and assume $W$ satisfies \hyp1 -- \hyp3. Then we have
    \begin{align} \label{NonLinInt_R}
        \lam_n \int_\R (u^{(n-1)})^2\dd x \leq \int_\R W(u)+(u^{(n)})^2\dd x    && \text{for all } u \in W^{n,2}_{\rm loc}(\R).
    \end{align}
\end{cor}

\subsection{Lower bound and compactness} \label{Subsec:LowBound}

Now, we use the nonlinear interpolation inequality from Theorem~\ref{Theo:NonLinInt_n} to estimate the energy $G_\eps^{\lam,n}$ from below by the related energy functional $G_\eps^{0,n}$ without the negative term. The latter has already been examined regarding compactness and $\Gamma$-convergence by Brusca et al. \cite{BruDonSol}, where the case $n=2$ was previously studied by Fonseca and Mantegazza \cite{FonMan}. The strategy for the estimation originates from \cite[Proposition~3.3]{CicSpaZep}.
 
 \begin{lemma}[Lower bound] \label{abschCic_n}
Let $n \in \N_{\geq2}$ and $\lam \in (0,\lam_n)$ and assume $W$ satisfies \hyp1--\hyp3. Let $I \subset \R$ be an open bounded interval and $\delta > 0$. Then there exists $\eps_0 > 0$ such that
    \begin{align} \label{eq:abschCic_n}
        \Big( 1 - \frac{\lam}{\lam_n} - \delta \Big) G_\eps^{0,n}[u] \leq G_\eps^{\lam,n}[u]   && \text{for all } u \in W^{n,2}(I), \ \eps \in (0, \eps_0).
    \end{align}
 \end{lemma}

 \begin{proof}
     We use the substitution $x \mapsto \eps x$, with $v(x) \coloneq  u(\eps x)$ and $\frac{I}{\eps} \coloneq \{x \in \mathbb{R}: \eps x \in I\}$. For every $u \in W^{n,2}(I)$, this results in the representation
    \begin{align}
        G^{\lam,n}_\eps[u]
        = \int_{\frac{I}{\eps}} W(v(x)) - \lam (v^{(n-1)}(x))^2 + (v^{(n)}(x))^2\dd x.
    \end{align}
     Next, we set $m_\eps \coloneq  \min\{m \in \mathbb{N}: m \geq \frac{|I|}{\eps}\}$ and subdivide the interval $\frac{I}{\eps}$ into $m_\eps$ disjoint open intervals $I_\eps^1, \ldots, I_\eps^{m_\eps}$ each of length $\frac{|I|}{\eps m_\eps} \eqcolon  l_\eps \leq 1$. Applying Theorem~\ref{Theo:NonLinInt_n} on each subinterval, we get
    \setlength\jot{0.5cm}
     \begin{align}
         G_\eps^{\lam,n}[u] & = \sum_{i=1}^{m_\eps} \int_{I_\eps^i}W(v) - \lam (v^{(n-1)})^2 + (v^{(n)})^2\dd x \\
         & \geq \sum_{i=1}^{m_\eps}  \int_{I_\eps^i} \! W(v)\dd x - \frac {\lam}{\lam_n}\bigg( \frac 1{l_\eps^{2n-2}}   \int_{I_\eps^i} \! W(v)\dd x + l_\eps^2  \int_{I_\eps^i}  (v^{(n)})^2 \dd x  \bigg) + \int_{I_\eps^i}  (v^{(n)})^2 \dd x \\
         & = \left(1-\frac{\lam}{\lam_n l_\eps^{2n-2}}\right) \int_{\frac{I}{\eps}} W(v)\dd x + \left(1-\frac{\lam l_\eps^2}{\lam_n}\right) \int_{\frac{I}{\eps}} (v^{(n)})^2\dd x.
    \end{align}
     The claim follows from resubstitution and $\lim_{\eps \to 0} l_\eps=\lim_{\eps \to 0} \frac{|I|}{\eps m_\eps} =1$, which is a direct consequence of the definition of $m_\eps$.
 \end{proof}

 The compactness property for $G_\eps^{\lam,n}$ (see Theorem~\ref{Theo:comp_n}) can now be derived from Lemma~\ref{abschCic_n} and the corresponding compactness property for $G_\eps^{0,n}$ proven by Brusca et al. for even more general double-well potentials (see {\cite[Proposition~5.1]{BruDonSol}}). However, they show compactness with respect to convergence in measure. In order to improve the result to $L^1$-convergence, we use Vitali's Convergence Theorem by employing the coercivity condition \hyp3. In contrast to the proof of Theorem~\ref{Theo:NonLinInt_n}, a linear coercivity condition would suffice here as well.

\begin{proof}[Proof of Theorem \ref{Theo:comp_n}.]
    According to Vitali's Convergence Theorem (see \cite[Satz 6.25]{Klenke}), the convergence of $u_\eps$ in $L^1(I)$ is equivalent to the convergence in measure if $u_\eps$ is uniformly integrable, that is, if it satisfies
        \begin{enumerate}
        \item[\eni] $\sup_{\eps} \|u_\eps\|_{L^1(I)} < \infty$ and
        \item[\enii] for every $\gamma>0$ there exists $\delta(\gamma)>0$ such that $\sup_\eps \|u_\eps\|_{L^1(J)} \leq \gamma$ for every measurable set $J \subset I$ with $|J| < \delta(\gamma)$
        \end{enumerate}
    (see \cite[Satz 6.24]{Klenke}). Moreover, the theorem states that the two limit functions coincide. Therefore, if we show the uniform integrability of $u_\eps$, we can deduce the claim from the compactness result \cite[Proposition 5.1]{BruDonSol}. Without loss of generality, after choosing a subsequence, we may consider $u_\eps \in W^{n,2}(I)$ with $\sup_\eps G_\eps^{\lam,n}[u_\eps] \leq \tilde{M} < \infty$.

    To show \eni, we use the coercivity condition \hyp3 and obtain
    \begin{align}
        W(t) \geq L(t \mp 1)^2 \geq L(\pm 2t-3) \quad \text{for all } \pm t \geq 0,
    \end{align}
    which is equivalent to
    \begin{align}
        \frac{W(t)}{2L} + \frac 32 \geq |t| \quad  \text{for all } t \in \R.
    \end{align}
    Together with Lemma~\ref{abschCic_n}, for sufficiently small $\eps$, this yields
    \begin{align} \label{viatliabsch}
    \begin{split}
        \int_I |u_\eps|\dd x
        & \leq \frac 1{2L} \int_I W(u_\eps)\dd x + \frac 32 |I| 
        \leq \frac 1{2L} \int_I W(u_\eps) + \eps^{2n} (u_\eps^{(n)})^2\dd x + \frac 32 |I| \\
        & = \frac \eps{2L} G_\eps^{0,n}[u_\eps] + \frac 32 |I|
        \leq \frac {\eps C}{2L} G_\eps^{\lam,n}[u_\eps] + \frac 32 |I|
        \leq \frac {\eps C}{2L} \tilde{M} + \frac 32 |I| \leq M < \infty,
    \end{split}
    \end{align}
    where $M$ only depends on $L$, $n$ and $\lam$.

    To show \enii, let $\gamma >0$, and choose $\eps_0 \coloneq \frac {L}{C\tilde{M}} \gamma$. Then, we find
    \begin{align}
        \frac {\eps C}{2L} \tilde{M} \leq \frac{\gamma}{2} \quad \text{for all } \eps \in (0,\eps_0).
    \end{align}
    Consequently, for sufficiently small $\eps$, if we choose $\delta \coloneq \frac \gamma3$ and reuse the estimate \eqref{viatliabsch}, we obtain
    \begin{align}
        \int_J |u_\eps|\dd x \leq \frac {\eps C}{2L} \tilde{M} + \frac 32 |J| \leq \frac \gamma2 + \frac \gamma2 = \gamma \quad \text{for } |J| \leq \delta,
    \end{align}
    which concludes the proof.
\end{proof}

In Lemma~\ref{abschCic_n}, a straightforward application of the nonlinear interpolation inequality from Theorem~\ref{Theo:NonLinInt_n} yielded the desired estimate by a known functional of the same order and therefore the compactness property we needed. However, it is also feasible to estimate $G^{\lam,n}_\eps$ by lower-order functionals using a technique by Chermisi et al. \cite[Section~4]{CheDalFonLeo}. In the following remark, we briefly discuss this alternative approach.

\begin{rem}[Alternative lower bound]
Instead of bounding $G^{\lam,n}_\eps$ from below by $G^{0,n}_\eps$ as in Lemma~\ref{abschCic_n}, it can also be estimated by $G^{0,n-1}_\eps$ of one order lower. To see this, we again use the representation
    \begin{align}
         G_\eps^{\lam,n}[u] & = \sum_{i=1}^{m_\eps} \int_{I_\eps^i}W(v) - \lam (v^{(n-1)})^2 + (v^{(n)})^2\dd x
    \end{align}
from the first part in the proof of Lemma~\ref{abschCic_n}. Then, we apply a technique from \cite[Section~4]{CheDalFonLeo} and write
    \begin{align} \label{etasumme}
    \begin{split}
        \frac{W(v)}{l_\eps^{2n-2}} - \lam(v^{(n-1)})^2+ l_\eps^2(v^{(n)})^2 = & \left(1-\frac{\lam_n-\lam}{\lam_n+1}\right) \! \left(\frac{W(v)}{l_\eps^{2n-2}}-\lam_n(v^{(n-1)})^2+l_\eps^2(v^{(n)})^2\right)\\
        &\ + \frac{\lam_n-\lam}{\lam_n+1} \left(\frac{W(v)}{l_\eps^{2n-2}} + (v^{(n-1)})^2+l_\eps^2(v^{(n)})^2\right).
    \end{split}
    \end{align}
According to Theorem~\ref{Theo:NonLinInt_n}, for every $i \in \{1,\ldots,m_\eps\}$ it holds
    \begin{align}
        0 \leq \frac{1}{l_\eps^{2n-2}} \int_{I_\eps^i} W(v)\dd x - \lam_n \int_{I_\eps^i} (v^{(n-1)})^2\dd x + l_\eps^2 \int_{I_\eps^i} (v^{(n)})^2\dd x.
    \end{align}
Hence, by integrating equation \eqref{etasumme} over $I_\eps^i$, we get
    \begin{align}
    \begin{split}
        \int_{I_\eps^i} \frac{W(v)}{l_\eps^{2n-2}} -\lam(v^{(n-1)})^2+ l_\eps^2 (v^{(n)})^2 \dd x 
        & \geq \frac{\lam_n-\lam}{\lam_n+1} \int_{I_\eps^i} \frac{W(v)}{l_\eps^{2n-2}} +(v^{(n-1)})^2+ l_\eps^2 (v^{(n)})^2\dd x \\
        &\geq \frac{\lam_n-\lam}{\lam_n+1} \int_{I_\eps^i} \frac{W(v)}{l_\eps^{2n-2}} + (v^{(n-1)})^2\dd x.
    \end{split}
    \end{align}
From $\lim_{\eps \to 0} l_\eps=\lim_{\eps \to 0} \frac{|I|}{\eps m_\eps} =1$ and by summing over all $I_\eps^i$, we finally obtain that for any $\delta > 0$ there exists $\eps_0 > 0$ such that for every $\lam \in (0, \lam_n)$, $\eps \in (0, \eps_0)$ and $u \in W^{n,2}(I)$, we have
    \begin{align}
        \left(\frac{\lam_n - \lam}{\lam_n + 1} - \delta\right) G_\eps^{0,n-1}[u] \leq G_\eps^{\lam,n}[u].
    \end{align}
In particular, it is possible to bound the second-order functional $G^{\lam,2}_\eps$ by the classical Cahn-Hilliard functional $G_\eps^{0,1}$, as done by Chermisi et al. \cite[Section~4]{CheDalFonLeo}. It is also conceivable to estimate the higher-order functional $G^{\lam,n}_\eps$ inductively by $G_\eps^{0,1}$ for $\lam > 0$ sufficiently small through an estimation chain of the type
 \begin{align}
     G_\eps^{0,1} \leq C_1 G_\eps^{\lam,2} \leq C_1 G_\eps^{0,2} \leq \ldots \leq C_{n-2} G_\eps^{\lam,n-1} \leq C_{n-2} G_\eps^{0,n-1} \leq C_{n-1} G_\eps^{\lam,n}.
\end{align}

\end{rem}

\section{Asymptotic Analysis} \label{Sec:Gamma}

In this chapter, we prove the main result stated in Theorem~\ref{Gamma_conv_n}. To show the Liminf inequality, we define an energetically favored competitor sequence for every sequence $u_\eps \in W^{n,2}(I)$ converging in $L^1(I)$. The exact construction is explained in Section~\ref{Subsec:FavComp}. These auxiliary sequences are piecewise defined, and must therefore be carefully examined for regularity. Thus, we present a technical lemma that addresses the coupling of the different segments of the sequence. In Section~\ref{Subsec:Gamman}, we finally show the Liminf inequality and the Limsup inequality, using the compactness property from Section~\ref{Subsec:LowBound} and the considerations from Section~\ref{Subsec:FavComp}.

\subsection{Construction of favored competitors} \label{Subsec:FavComp}

 To prove the Liminf inequality, we consider sequences $u_\eps \in W^{n,2}(I)$ with $u_\eps \to u$ in $L^1(I)$ and $u \in BV(I,\{\pm1\})$. For each of these sequences, we construct a favored competitor sequence $v_\eps \in W^{n,2}(I)$ such that
\begin{align} \label{competitor}
    G_\eps^{\lam,n}[u] \leq \liminf_{\eps \to 0} G_\eps^{\lam,n}[v_\eps] \leq \liminf_{\eps \to 0} G_\eps^{\lam,n}[u_\eps].
\end{align}
 The idea of the construction is to modify the sequence $u_\eps$ so that it equates the limit function $u$ away from its points of discontinuity $s_1, \ldots , s_N \in S(u)$. To ensure that the sequence $v_\eps$ is a sequence of test functions for the optimal profile problem $C_W^{\lam,n}$ (see \eqref{Gammaconst} below), the coupling of $u_\eps$ to the constant segments must guarantee the competitor's absolute continuity of all derivatives up to order $n-1$. For certain $y_0,\ldots,y_{n-1} \in \R$, these couplings are given by rescaled versions of functions $\zeta \in C^n([0,1])$ satisfying
     \begin{align}
        \zeta(0)=y_0 \quad &\text{and} \quad \zeta(1)=1, \label{bed1_zeta} \\
        \zeta^{(k)}(0)=y_k \quad &\text{and} \quad \zeta^{(k)}(1)=0 \text{ for } k \in \{1,\ldots,n-1\}, \label{bed2_zeta}
     \end{align}
and functions $\eta \in C^n([0,1])$ satisfying
    \begin{align}
        \eta(0)=-1 \quad &\text{and} \quad \eta(1)=y_0, \label{bed1_eta} \\
        \eta^{(k)}(0)=0 \quad &\text{and} \quad \eta^{(k)}(1)=y_k \text{ for } k \in \{1,\ldots,n-1\}, \label{bed2_eta},
     \end{align}
respectively, depending on whether it is a coupling to $1$ or $-1$. In the subsequent considerations, we will choose $y_0,\ldots,y_{n-1}$ as the corresponding evaluations of the derivatives of $u_\eps$ to ensure the regularity of the favored competitor. We define the sets
    \begin{align}
        \AAA^n_I(y)&\coloneq \{\zeta \in C^n([0,1]): \zeta \text{ satisfies \eqref{bed1_zeta} and \eqref{bed2_zeta}}\}, \label{AI}\\
        \AAA^n_J(y)&\coloneq \{\eta \in C^n([0,1]): \eta \text{ satisfies \eqref{bed1_eta} and \eqref{bed2_eta}}\}, \label{AJ}
    \end{align}
and justify that they are not empty by showing the existence of polynomials satisfying the respective conditions.

 \begin{lemma}[Existence of a coupling] \label{poly}
     Let $n \in \N_{\geq2}$ and $y \coloneq (y_0,\ldots,y_{n-1}) \in \R^n$. There exist polynomials $p_n \in \AAA^n_I(y)$ and $q_n \in \AAA^n_J(y)$ of degree $N \leq 2n-1$.
 \end{lemma}
 
 \begin{proof}
     We first establish the case $p_n \in \AAA^n_I(y)$. In general, polynomials of degree $N \leq 2n-1$ and their derivatives are of the form
     \begin{align}
         p_n(x) = \sum_{i=0}^{2n-1} a_i \, x^i, \qquad p_n^{(k)}(x) = \sum_{i=0}^{2n-1-k} \frac{(k+i)!}{i!}\, a_{k+i}\, x^i \quad\text{for } k \in \{1,\ldots,n-1\},
    \end{align}
    where $a_i \in \R$ for $i \in \{0,\ldots,2n-1\}$. With this representation, \eqref{bed1_zeta} and \eqref{bed2_zeta} read as
    \begin{align}
        a_0=y_0 \quad &\text{and} \quad \sum_{i=0}^{2n-1} a_i = 1, \\
        k!\,a_k = y_k \quad &\text{and} \quad \sum_{i=0}^{2n-1-k} \frac{(k+i)!}{i!}\, a_{k+i} = 0 \quad \text{for } k \in \{1,\ldots,n-1\}.
    \end{align}
    Consequently, finding a polynomial satisfying \eqref{bed1_zeta} and \eqref{bed2_zeta} is equivalent to solving a system of linear equations given by the matrix equation
    \begin{align} \label{GLS}
        \renewcommand{\arraystretch}{1.1}
        \Biggl( \begin{array}{rr}
            \mala{A}   &   \mala{0}   \\
            \mala{B}   &   \mala{C}
        \end{array} \Biggr)
        \big(\ a_0 \ \cdots \ a_{2n-1}\ \big)^T
        =
        \big(\ y_0 \ \cdots \ y_{n-1} \ 1 \ 0 \ \cdots \ 0\ \big)^T
    \end{align}
    where the coefficient matrix is a lower block triangular. The submatrices $A=(A_{ij}),\,\allowbreak{B=(B_{ij})},\,\allowbreak C=(C_{ij}) \in \R^{n \times n}$ are defined as
    \begin{align}
        A_{ij}= \left\{ \begin{array}{ll} (i-1)!, & i=j, \\    0, & i \neq j, \end{array} \right. \qquad
        B_{ij}= \left\{ \begin{array}{ll} \frac{(n+j-1)!}{(n+j-i)!} & i\leq j, \\    0, & i > j, \end{array} \right. \qquad
        C_{ij}=\frac{(n+j-1)!}{(n+j-i)!}
    \end{align}
    for all $i,j \in \{1,\ldots,n\}$. The system of equations \eqref{GLS} is uniquely solvable if and only if the determinant of the coefficient matrix is non-zero. This determinant is given by the product of the determinants of the submatrices $A$ and $C$ (see \cite[Section~2]{Sil}). Since $A$ is a diagonal matrix, we can easily calculate
    \begin{align}
        \det(A)=\prod_{i=1}^n (i-1)! > 0.
    \end{align}
    In order to determine the determinant of $C$, we write
    \begin{align}
        C_{ij}=\frac{(n+j-1)!}{(n+j-i)!}=(i-1)! \binom{n+j-1}{i-1}.
    \end{align}
    For $D=(D_{ij})\in \R^{n \times n}$ given by $D_{ij}=\binom{n+j-1}{i-1}$, we have $\det(D)=1$ (see \cite[Paragraph~167]{Net}), and thus
    \begin{align}
        \det(C)=\prod_{i=1}^n (i-1)! \det(D)
        =\prod_{i=1}^n (i-1)! > 0.
    \end{align}
    Hence, the coefficient matrix is invertible and \eqref{GLS} is uniquely solvable, so that we can conclude that there exists a polynomial $p_n$ of degree $N \leq 2n-1$ satisfying \eqref{bed1_zeta} and \eqref{bed2_zeta}. The polynomial $q_n$ can be obtained through translation and reflection of $p_n$.
 \end{proof}
 \begin{rem} \label{contpol}
     We note that the admissible polynomials in Lemma~\ref{poly} are uniquely determined if we require that the degree be no greater than $2n-1$. In addition, their coefficients continuously depend on the data $y \in \R^n$.
 \end{rem}

\subsection{\texorpdfstring{$\Gamma$}{Gamma}-convergence} \label{Subsec:Gamman}

 Equipped with the compactness property from Theorem~\ref{Theo:comp_n} and the investigations from Section~\ref{Subsec:FavComp}, we now separately prove the Liminf inequality and the Limsup inequality. However, we first address the well-definedness of the $\Gamma$-limit functional \eqref{Gammalimes}. Specifically, we show that the optimal profile problem
 \begin{align} \label{Gammaconst}
    C_W^{\lam,n} \coloneq  \inf\left\{\int_\R W(f) - \lam(f^{(n-1)})^2 + (f^{(n)})^2\dd x : f \in \AAA^n(\R) \right\},
\end{align}
where $\AAA^n(\R) \coloneq  \left\{ f \in W^{n,2}_{\mathrm{loc}}(\R) \colon f(x)=1 \text{ for } x > T \text{ and } f(x) = -1 \text{ for } x < -T  \text{ for some } T>0\right\}$, defines a positive constant. To this end, we refer to the study of the optimal profile problems in \cite[Section~4]{BruDonSol} and the estimate from Lemma~\ref{abschCic_n}.

\begin{prop}[Optimal Profile on the real line] \label{constwohl}
    Let $n \in \N_{\geq2}$, $\lam \in (0,\lam_n)$ and assume $W$ satisfies \hyp1 -- \hyp3. Then $C^{\lam,n}_W > 0$ is well-defined.
\end{prop}
\begin{proof}
     We show the claim by comparing it with the optimal profile problem
    \begin{align}
        C_W^n \coloneq \inf\left\{\int_\R W(f) + (f^{(n)})^2\dd x : f \in \AAA^n(\R) \right\},
    \end{align}
    which has already been shown to be a positive constant (see \cite[Section~4]{BruDonSol}). By Corollary~\ref{corNonLinInt_R}, we have
    \begin{align}
        \left(1-\frac{\lam}{\lam_n}\right) \int_{\R} W(f) +(f^{(n)})^2\dd x
        \leq \int_{\R} W(f) - \lam (f^{(n-1)})^2 + (f^{(n)})^2\dd x,
    \end{align}
    which yields the estimate 
    \begin{align}
        \left(1-\frac{\lam}{\lam_n}\right) C^n_W \leq C_W^{\lam,n} \leq C^n_W
    \end{align}
    and therefore the positivity and well-definedness of $C_W^{\lam,n}$.
\end{proof}

 Knowing that \eqref{Gammaconst} defines a positive constant, we now show that the functional $G^{\lam,n}$ defined by \eqref{Gammalimes} is an asymptotic lower bound for the sequence $G^{\lam,n}_\eps$, thus proving the validity of the Liminf inequality. As described in Section~\ref{Subsec:FavComp}, the strategy consists of finding a favored competitor sequence $v_\eps \in W^{n,2}(I)$ for every sequence $u_\eps \in W^{n,2}(I)$ that converges in $L^1(I)$, and proving \eqref{competitor}. Similar methods for the related functionals $G^{\lam,2}_\eps$ and $G^{0,2}_\eps$ are employed in \cite[Theorem~4.1]{CicSpaZep} and \cite[Proposition~2.7]{FonMan}, respectively.

\begin{prop}[Liminf inequality] \label{prop:liminf_inequality}
    Let $n \in \N_{\geq2}$ and $\lam \in (0,\lam_n)$. For every sequence ${u_\eps \in L^1(I)}$ and $u\in L^1(I)$ such that $u_\eps \to u$ in $L^1(I)$ as $\eps \to 0$, we have
    \begin{align}
        G^{\lam,n}[u] \leq \liminf_{\eps \to 0} G^{\lam,n}_\eps[u_\eps].
    \end{align}
\end{prop}

\begin{proof} Without loss of generality, let $u_\eps \in W^{n,2}(I)$ and
    \begin{align} \label{boundedseq_n}
       \liminf_{\eps \to 0} G_\eps^{\lam,n}[u_\eps] \leq C < \infty.
    \end{align}
     Exploiting Theorem~\ref{Theo:comp_n}, we conclude $u \in BV(I,\{\pm1\})$. Now, let $S(u)\coloneq \{s_1,\ldots,s_N\}$ with $s_1<\ldots<s_N$ be the set of discontinuity points of $u$, which is finite due to the bounded variation of $u$, and $\delta_0\coloneq \min\{s_{i+1}-s_i \colon 1\leq i \leq N-1\}$ the minimal distance between two adjacent points in $S(u)$. We choose a fixed $\delta \in (0,\frac{\delta_0}{2})$. Up to subsequences, we have $u_\eps \to u$ pointwise almost everywhere in $I$ as $\eps \to 0$.
     
    \begin{figure}[htb]
\centering
\begin{tikzpicture}[scale=1]
\draw[shift={(0,1)}] (2pt,0pt) -- (-2pt,0pt) node[below=0.5em] {};
\draw[shift={(0,-1)}] (2pt,0pt) -- (-2pt,0pt) node[below=0.5em] {};
\draw[->] (-6,0) -- (6,0) node[right] {$x$};
\draw[->] (0,-2) -- (0,2);% node[above] {\textcolor{blue}{$u(x)$}};
\draw[-,blue] (-5.5,-1) -- (-4,-1) -- (-4,1) -- (-3,1) -- (-3,-1) -- (-1.5,-1) -- (-1.5,1) -- (1,1) -- (1,-1) -- (2,-1) -- (2,1) -- (4,1) -- (4,-1) -- (5.5,-1);
\draw[violet] (-3.5,0) node[xshift=-2pt] {\textbf{$)$}};
\draw[violet] (-4,0) node[yshift=-1.4cm] {\textbf{$s_1$}};
\node at (-4,0) [circle,fill=violet,inner sep=1.2pt]{};
\draw[violet] (-2.5,0) node[xshift=-2pt] {\textbf{$)$}};
\draw[violet] (-3,0) node[yshift=-1.4cm] {\textbf{$s_2$}};
\node at (-3,0) [circle,fill=violet,inner sep=1.2pt]{};
\draw[violet] (-1,0) node[xshift=-2pt] {\textbf{$)$}};
\draw[violet] (-1.5,0) node[yshift=-1.4cm] {\textbf{$s_3$}};
\node at (-1.5,0) [circle,fill=violet,inner sep=1.2pt]{};
\draw[violet] (1.5,0) node[xshift=-2pt] {\textbf{$)$}};
\draw[violet] (1,0) node[yshift=-1.4cm] {\textbf{$s_4$}};
\node at (1,0) [circle,fill=violet,inner sep=1.2pt]{};
\draw[violet] (2.5,0) node[xshift=-2pt] {\textbf{$)$}};
\draw[violet] (2,0) node[yshift=-1.4cm] {\textbf{$s_5$}};
\node at (2,0) [circle,fill=violet,inner sep=1.2pt]{};
\draw[violet] (4.5,0) node[xshift=-2pt] {\textbf{$)$}};
\draw[violet] (4,0) node[yshift=-1.4cm] {\textbf{$s_6$}};
\node at (4,0) [circle,fill=violet,inner sep=1.2pt]{};
\draw[violet] (-4.5,0) node[xshift=2pt] {\textbf{$($}};
\draw[violet] (-3.5,0) node[xshift=2pt] {\textbf{$($}};
\draw[violet] (-2,0) node[xshift=2pt] {\textbf{$($}};
\draw[violet] (0.5,0) node[xshift=2pt] {\textbf{$($}};
\draw[violet] (1.5,0) node[xshift=2pt] {\textbf{$($}};
\draw[violet] (3.5,0) node[xshift=2pt] {\textbf{$($}};
\node[orange!85!white]  at (-1.3,1.2) {\textcolor{blue}{$u$}};
\end{tikzpicture}
\caption{We construct non-intersecting balls around the points of discontinuity of $u$.}
\end{figure}
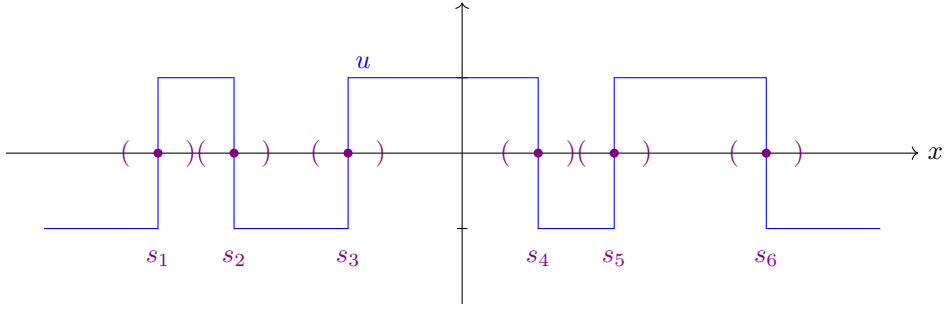%
    
    \noindent Moreover, using Lemma~\ref{abschCic_n} and choosing $\eps>0$ sufficiently small, we have
    \begin{align}
        \eps^{2n-1} \|u_\eps^{(n)}\|^2_{L^2(I)} 
        & \overset{\eqref{boundedseq_n}}{\leq} C+ \int_I \lam\eps^{2n-3} (u_\eps^{(n-1)})^2 - \frac 1\eps W(u_\eps)\dd x \\
        & \overset{\eqref{eq:abschCic_n}}{\leq} C + \left(\frac{\lam}{\lam_n} + \delta\right) \eps^{2n-1} \|u_\eps^{(n)}\|_{L^2(I)}^2 
    \end{align}
    for all $u_\eps \in W^{n,2}(I)$ and $\lam \in (0,\lam_n)$, which yields $\|u_\eps^{(n)}\|_{L^2(I)} \leq C \eps^{\frac 12 -n}$ by absorption. With this estimate and a suitable formulation of the Gagliardo-Nirenberg inequality given in Theorem~\ref{GagNir}, we infer that the derivatives up to order $n-1$ are also converging. More precisely, choosing the parameters $p = \frac{2n}{2n - k}$, $r = 2$, $q = 1$ and $\theta = \frac kn$, we obtain
        \begin{align}
            \|u_\eps^{(k)}\|_{L^p(I)} \leq C \left(\|u_\eps^{(n)}\|_{L^2(I)}^{\frac kn} \|u_\eps\|_{L^1(I)}^{1 - \frac kn} + \|u_\eps\|_{L^1(I)} \right) \leq C\left(\eps^{k\frac{1 - 2n}{2n}} + 1\right),
        \end{align}
    from which we find, using Hölder's inequality, $\eps^\alpha \|u_\eps^{(k)}\|_{L^1(I)} \to 0$ as $\eps \to 0$ for $\alpha > k \frac{2n - 1}{2n}$. In particular, we find yet again a subsequence (not relabeled) such that $\eps^{k} u^{(k)}_\eps \to 0$ pointwise almost everywhere in $I$ as $\eps \to 0$. Hence, for every $\sigma > 0$ we find $\eps_0 = \eps_0(\sigma) > 0$ and two points $x_{\sigma,i}^+,x_{\sigma,i}^- \in B_\delta(s_i)\cap I$ such that for $\eps  \in (0,\eps_0)$ we have
    \begin{align} \label{convlower}
      |u_\eps(x_{\sigma,i}^\pm)\mp 1|<\sigma \quad \text{and} \quad
      |\eps^{k} u^{(k)}_\eps(x_{\sigma,i}^\pm)|<\sigma \quad \text{for } k \in \{1,\ldots, n-1\}. 
    \end{align}
    Given $y \coloneq (y_0,\ldots,y_{n-1}) \in \R^n$, we define $I^{\lam,n}, J^{\lam,n}: \R^n \to \R$ by
    \begin{align}
        I^{\lam,n}(y)&\coloneq \inf \Big\{ \int_0^1 W(\zeta)-\lam(\zeta^{(n-1)})^2+(\zeta^{(n)})^2\dd x: \zeta \in \AAA^n_I(y) \Big\},\\
        J^{\lam,n}(y)&\coloneq \inf \Big\{ \int_0^1 W(\eta)-\lam(\eta^{(n-1)})^2+(\eta^{(n)})^2\dd x: \eta \in \AAA^n_J(y) \Big\},
    \end{align}
    where $\AAA^n_I(y)$ and $\AAA^n_J(y)$ are defined in \eqref{AI} and \eqref{AJ}, respectively. Testing with admissible polynomials $p_n \in \AAA^n_I(y)$ and $q_n \in \AAA^n_J(y)$ of degree $N \leq 2n-1$, whose existence is ensured by Lemma~\ref{poly}, yields
    \begin{align}
        0 \leq I^{0,n}(y) \leq \int_0^1 W(p_n)+(p_n^{(n)})^2 \dd x  \quad \text{and} \quad  0 \leq J^{0,n}(y) \leq \int_0^1 W(q_n)+(q_n^{(n)})^2 \dd x.
    \end{align}
    Moreover, we observe that $\tilde{p}_n \in \AAA^n_I(e_1)$ and $\tilde{q}_n \in \AAA^n_J(-e_1)$ are given by constant functions, where $e_1$ denotes the first standard unit vector of $\R^n$. From \hyp1 and the continuous dependence of the polynomials on $y$ (see Remark~\ref{contpol}), we hence conclude
    \begin{align}
        \lim_{y \to e_1} I^{0,n}(y) =  \lim_{y \to -e_1} J^{0,n}(y) = 0.
    \end{align}
    From Theorem~\ref{Theo:NonLinInt_n}, it follows
    \begin{align}
        \left(1- \frac \lam{\lam_n}\right)I^{0,n} \leq I^{\lam,n} \leq I^{0,n} \quad \text{and} \quad \left(1- \frac \lam{\lam_n}\right)J^{0,n} \leq J^{\lam,n} \leq J^{0,n}
    \end{align}
    for all $\lam \in (0,\lam_n)$, thereby also obtaining
    \begin{align} \label{IlamJlam_n}
        \lim_{y \to e_1} I^{\lam,n}(y) =  \lim_{y \to -e_1} J^{\lam,n}(y) = 0.
    \end{align}
    Without loss of generality, we assume $x_{\sigma,i}^- < x_{\sigma,i}^+$ and define $y^\pm_{\eps,\sigma,i} \in \R^n$ component-wise by
    \begin{align}
        (y^{\pm}_{\eps,\sigma,i})_k \coloneq \left\{\begin{array}{ll}
    u_\eps(x_{\sigma,i}^\pm) & \quad \text{$k=0$}, \\
    \eps^{k} u^{(k)}_\eps(x_{\sigma,i}^\pm) & \quad \text{$k \in \{1,\ldots,n-1\}$}.
    \end{array}\right.
    \end{align}
    We note that $y_{\eps,\sigma,i}^\pm \to \pm e_1$ as $\eps \to 0$ due to \eqref{convlower}. Let $\zeta_{\eps,\sigma,i} \in \AAA^n_I(y^+_{\eps,\sigma,i})$ and $\eta_{\eps,\sigma,i} \in \AAA^n_J(y^-_{\eps,\sigma,i})$ such that
    \begin{align} \label{abschI_lamJ_lam_n}
    \begin{split}
        \int_0^1 W(\zeta_{\eps,\sigma,i})-\lam(\zeta^{(n-1)}_{\eps,\sigma,i})^2+(\zeta^{(n)}_{\eps,\sigma,i})^2\dd x &\leq I^{\lam,n}(y^+_{\eps,\sigma,i}) + \frac \sigma{2N},\\
        \int_0^1 W(\eta_{\eps,\sigma,i})-\lam(\eta^{(n-1)}_{\eps,\sigma,i})^2+(\eta^{(n)}_{\eps,\sigma,i})^2\dd x &\leq J^{\lam,n}(y^-_{\eps,\sigma,i}) + \frac \sigma{2N},
    \end{split}
    \end{align}
    and set
    \begin{align}
        \hat{\zeta}_{\eps,\sigma,i}(x)\coloneq \zeta_{\eps,\sigma,i}\left(x-\frac{x_{\sigma,i}^+}{\eps}\right) \quad \text{and} \quad \hat{\eta}_{\eps,\sigma,i}(x)\coloneq \eta_{\eps,\sigma,i}\left(x-\frac{x_{\sigma,i}^-+1}{\eps}\right).
    \end{align}
    Thus, we define the favored competitor sequence of functions $v_{\eps,i} \in \AAA^n(\R)$ as  
    \begin{align}
        v_{\eps,\sigma,i}(x)\coloneq 
        \left\{\begin{array}{ll}
        -1 \quad & x \leq \frac{x_{\sigma,i}^-}{\eps} -1, \\
        \hat{\eta}_{\eps,\sigma,i} \quad & \frac{x_{\sigma,i}^-}{\eps}-1 \leq x \leq \frac{x_{\sigma,i}^-}{\eps}, \\
        u_\eps (\eps x) \quad & \frac{x_{\sigma,i}^-}{\eps}\leq x \leq \frac{x_{\sigma,i}^+}{\eps}, \\
        \hat{\zeta}_{\eps,\sigma,i} \quad & \frac{x_{\sigma,i}^+}{\eps} \leq x \leq \frac{x_{\sigma,i}^+}{\eps}+1, \\
        1 \quad & \frac{x_{\sigma,i}^+}{\eps} +1 \leq x.
        \end{array}\right.
    \end{align} 
    
    \vspace{-0.3cm}
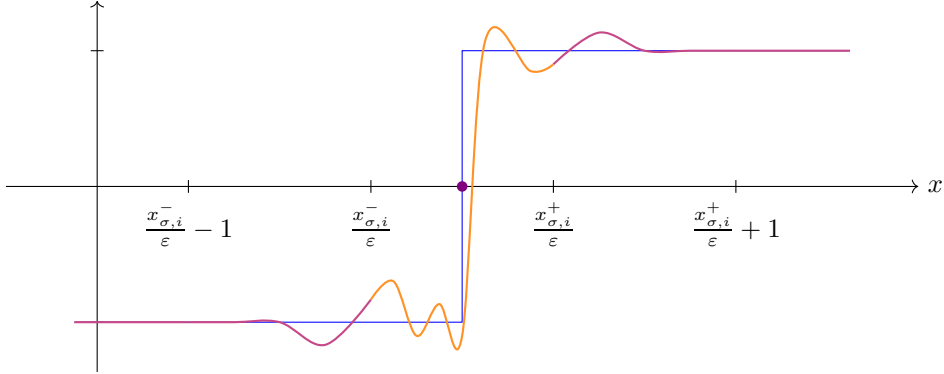
\begin{figure}[htb]
\centering   
\begin{tikzpicture}[scale=0.6]
\draw[->] (-10,0) -- (10,0) node[right] {$x$};
\draw[->] (-8,-4.1) -- (-8,4.1);% node[above] {\textcolor{blue}{$u(x)$}};
\draw[shift={(-8,3)}] (4pt,0pt) -- (-4pt,0pt) node[below=0.5em] {};
\draw[shift={(-8,-3)}] (4pt,0pt) -- (-4pt,0pt) node[below=0.5em] {};
\draw[-,blue] (-8.5,-3) -- (0,-3) -- (0,3) -- (8.5,3);

\node at (0,0) [circle,fill=violet,inner sep=1.4pt]{};

\draw[shift={(-2,0)}] (0pt,-4pt) -- (0pt,4pt) node[below=0.5em] {$\frac{x_{\sigma,i}^-}{\eps}$};
\draw[shift={(-6,0)}] (0pt,-4pt) -- (0pt,4pt) node[below=0.5em] {$\frac{x_{\sigma,i}^-}{\eps}-1$};

\draw[shift={(2,0)}] (0pt,-4pt) -- (0pt,4pt) node[below=0.5em] {$\frac{x_{\sigma,i}^+}{\eps}$};
\draw[shift={(6,0)}] (0pt,-4pt) -- (0pt,4pt) node[below=0.5em] {$\frac{x_{\sigma,i}^+}{\eps}+1$};

\draw[thick,-,magenta!80!black] (-8.5,-3) -- (-6,-3);
\draw[thick,-,magenta!80!black] (8.5,3) -- (6,3);

\draw[magenta!80!black,thick]  plot[smooth, tension=0.6] coordinates {(-6,-3) (-5,-3) (-4,-3) (-3,-3.5) (-2, -2.5)};
\draw[magenta!80!black,thick]  plot[smooth, tension=0.6] coordinates {(6,3) (5,3) (4,3) (3,3.4) (2, 2.7)};

\draw[orange!85!white,thick]  plot[smooth, tension=0.6] coordinates {(-2, -2.5) (-1.5, -2.1) (-1, -3.3) (-0.5, -2.6) (0, -3.3) (0.5, 3.2) (1.5, 2.55) (2, 2.7)};

\end{tikzpicture}
\caption{Away from the discontinuity, the sequence $u_\eps$ (orange) is modified (pink) into a favored competitor sequence $v_\eps$ such that its energy $G_\eps^{\lam,n}$ decreases.}
\end{figure}%
    \vspace{12pt}
        
    \noindent Using $W(\pm 1)=0$ and \eqref{abschI_lamJ_lam_n}, we have     
    \begin{align} \label{abschOptimalProf}
    \begin{split}
        C^{\lam,n}_W & \leq \int_{\frac{x_{\sigma,i}^-}{\eps} -1}^{\frac{x_{\sigma,i}^+}{\eps} +1} W(v_{\eps,\sigma,i})-\lam(v^{(n-1)}_{\eps,\sigma,i})^2+(v^{(n)}_{\eps,\sigma,i})^2\dd x\\
        & \leq \int_{x_{\sigma,i}^-}^{x_{\sigma,i}^+} \frac 1\eps W(u_\eps)-\lam \eps^{2n-3} (u^{(n-1)}_\eps)^2 + \eps^{2n-1}(u^{(n)}_\eps)^2\dd x + J^{\lam,n}(y^-_{\eps,\sigma,i}) + I^{\lam,n}(y^+_{\eps,\sigma,i}) + \frac \sigma N.
    \end{split}
    \end{align}
    Finally, we get
    \begin{align}
        \liminf_{\eps \to 0} G_\eps^{\lam,n}[u_\eps] & \overset{\hphantom{\eqref{abschOptimalProf}}}{\geq} \liminf_{\eps \to 0} \sum_{i=1}^N \int_{x_{\sigma,i}^-}^{x_{\sigma,i}^+} \frac 1\eps W(u_\eps)-\lam \eps^{2n-3} (u^{(n-1)}_\eps)^2 + \eps^{2n-1}(u^{(n)}_\eps)^2\dd x \\
        & \overset{\eqref{abschOptimalProf}}{\geq} N C_W^{\lam,n} - \limsup_{\eps \to 0} \sum_{i=1}^N \left( J^{\lam,n}(y^-_{\eps,\sigma,i}) + I^{\lam,n}(y^+_{\eps,\sigma,i}) + \frac \sigma N \right) = N C^{\lam,n}_W - \sigma,
    \end{align}
    where we have used $y_{\eps,\sigma,i}^\pm \to \pm e_1$ as $\eps \to 0$ and \eqref{IlamJlam_n}. The Liminf Inequality follows by letting $\sigma \to 0$.
    \end{proof}

 It remains to show that the asymptotic lower bound $G^{\lam,n}$ is optimal as described by the Limsup inequality. This is demonstrated through the existence of a recovery sequence, which we define explicitly in the following. The approach is a standard method in the context of singular perturbation models and is also used for $G^{\lam,2}_\eps$ and $G^{0,n}_\eps$ in \cite[Theorem~4.1]{CicSpaZep} and \cite[Proposition 5.3]{BruDonSol}, respectively.

\begin{prop}[Limsup inequality] \label{prop:limsup_inequality}
    Let $n \in \N_{\geq2}$ and $\lam \in (0,\lam_n)$. For every $u\in L^1(I)$ there exists a sequence $u_\eps \in L^1(I)$ such that $u_\eps \to u$ in $L^1(I)$ as $\eps \to 0$ and
    \begin{align}
        G^{\lam,n}[u] \geq \limsup_{\eps \to 0} G^{\lam,n}_\eps[u_\eps].
    \end{align}
\end{prop}

\begin{proof}
    We denote $I=(a,b)$ and consider $u \in BV(I, \{\pm 1\})$. In addition to the setting in the proof of Proposition~\ref{prop:liminf_inequality}, we set $s_0\coloneq a$ and $s_{N+1}\coloneq b$. For all $i \in \{1,\ldots,N\}$, we define the closed intervals $I_i\coloneq \left[\frac 12 (s_{i-1}+s_i),\frac 12 (s_i+s_{i+1})\right]$, and $\delta_0\coloneq \min\{s_{i+1}-s_i \colon 0\leq i\leq N\}$. Further we may assume that jumps of $u$ from $-1$ to $1$ occur in intervals with odd indices and jumps of $u$ from $1$ to $-1$ occur in intervals with even indices. 
    
    \begin{figure}[htb]
\centering
\begin{tikzpicture}[scale=0.9]
\draw[shift={(0,1.5)}] (2pt,0pt) -- (-2pt,0pt) node[below=0.5em] {};
\draw[shift={(0,-1.5)}] (2pt,0pt) -- (-2pt,0pt) node[below=0.5em] {};
\draw[->] (-7,0) -- (7,0) node[right] {$x$};
\draw[->] (0,-2.5) -- (0,2.5);% node[above] {\textcolor{blue}{$u(x)$}};
\draw[-,blue] (-6,-1.5) -- (-4,-1.5) -- (-4,1.5) -- (-1,1.5) -- (-1,-1.5) -- (2,-1.5) -- (2,1.5) -- (4,1.5) -- (4,-1.5) -- (6,-1.5);

\draw[violet] (-6,0) node[xshift=-1pt] {\textbf{\Large{$($}}};
\draw[violet] (-6,0) node[yshift=-1.9cm] {\textbf{$s_0$}};
\draw[violet] (-5,0) node[xshift=-0pt] {\textbf{$]$}};

\draw[violet] (-5,0) node[xshift=0pt] {\textbf{$[$}};
\node at (-4,0) [circle,fill=violet,inner sep=1.2pt]{};
\draw[violet] (-4,0) node[yshift=-1.9cm] {\textbf{$s_1$}};
\draw[violet] (-2.5,0) node[xshift=0pt] {\textbf{$]$}};

\draw[violet] (-2.5,0) node[xshift=0pt] {\textbf{$[$}};
\node at (-1,0) [circle,fill=violet,inner sep=1.2pt]{};
\draw[violet] (-1,0) node[yshift=-1.9cm] {\textbf{$s_2$}};
\draw[violet] (0.5,0) node[xshift=0pt] {\textbf{$]$}};

\draw[violet] (0.5,0) node[xshift=0pt] {\textbf{$[$}};
\node at (2,0) [circle,fill=violet,inner sep=1.2pt]{};
\draw[violet] (2,0) node[yshift=-1.9cm] {\textbf{$s_3$}};
\draw[violet] (3,0) node[xshift=0pt] {\textbf{$]$}};

\draw[violet] (3,0) node[xshift=0pt] {\textbf{$[$}};
\node at (4,0) [circle,fill=violet,inner sep=1.2pt]{};
\draw[violet] (4,0) node[yshift=-1.9cm] {\textbf{$s_4$}};
\draw[violet] (5,0) node[xshift=0pt] {\textbf{$]$}};

\draw[violet] (6,0) node[yshift=-1.9cm] {\textbf{$s_5$}};
\draw[violet] (6,0) node[xshift=-1pt] {\textbf{\Large{$)$}}};

\draw[thick,-,magenta!80!black] (-6,-1.5) -- (-5,-1.5);

\draw[magenta!80!black,thick]  plot[smooth, tension=0.5] coordinates {(-5,-1.5) (-4.8,-1.5) (-4.5,-1.45) (-4.2,-1.55) (-3.8,1.55) (-3.5,1.45) (-3.2,1.5) (-3,1.5)};

\draw[thick,-,magenta!80!black] (-3,1.5) -- (-2,1.5);

\draw[magenta!80!black,thick]  plot[smooth, tension=0.5] coordinates {(0,-1.5) (-0.2,-1.5) (-0.5,-1.45) (-0.8,-1.55) (-1.2,1.55) (-1.5,1.45) (-1.8,1.5) (-2,1.5)};

\draw[thick,-,magenta!80!black] (0,-1.5) -- (1,-1.5);

\draw[magenta!80!black,thick]  plot[smooth, tension=0.5] coordinates {(1,-1.5) (1.2,-1.5) (1.5,-1.45) (1.8,-1.55) (2.2,1.55) (2.5,1.45) (2.8,1.5) (3,1.5)};

\draw[magenta!80!black,thick]  plot[smooth, tension=0.5] coordinates {(5,-1.5) (4.8,-1.5) (4.5,-1.45) (4.2,-1.55) (3.8,1.55) (3.5,1.45) (3.2,1.5) (3,1.5)};

\draw[thick,-,magenta!80!black] (6,-1.5) -- (5,-1.5);
\end{tikzpicture}
\caption{As $\eps$ goes to zero, the recovery sequence $u_\eps$ (pink) converges to $u$ (blue).}
\end{figure}
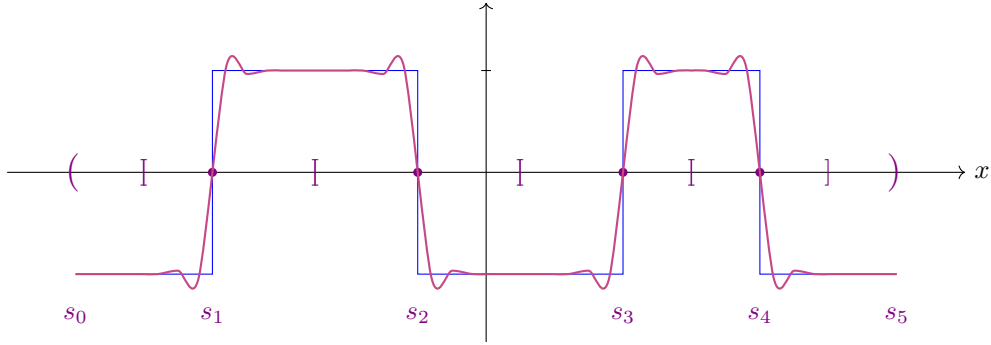%
    
    \noindent We then fix $\delta \in (0,\delta_0)$ and $f \in \AAA^n$ such that 
    \begin{align} \label{choosef_n}
        \int_\R W(f)-\lam(f^{(n-1)})^2+(f^{(n)})^2\dd x \leq C^{\lam,n}_W + \frac{\delta}{N}.
    \end{align}    
     Moreover, we choose $\eps_0>0$ sufficiently small such that we have $\frac{\delta}{2\eps}>T$ for all $\eps \in (0,\eps_0)$, and define the recovery sequence $u_\eps \in W^{n,2}(I)$ by
    \begin{align}
        u_\eps(x)\coloneq 
        \left\{\begin{array}{ll}
        f\left(\frac{x-s_i}{\eps}\right), \quad & x \in I_i \text{ and $i$ is odd}, \\
        f\left(-\frac{x-s_i}{\eps}\right), \quad & x \in I_i \text{ and $i$ is even}, \\
        u(x), \quad & \text{otherwise}.
        \end{array}\right.
    \end{align}
    Due to $f \in \AAA^n$, we hence have $u_\eps \to u$ in $L^1(I)$ as $\eps \to 0$, and
        \begin{align}
            \lim_{\eps \to 0} G_\eps^{\lam,n}[u_\eps] 
            & \overset{\hphantom{\eqref{choosef_n}}}{=} \lim_{\eps \to 0} \sum_{i=1}^N \int_{I_i} \frac{1}{\eps} W(u_\eps) -\lam\eps^{2n-3}(u^{(n-1)}_\eps)^2+\eps^{2n-1}(u^{(n)}_\eps)^2\dd x \\
            &\overset{\hphantom{\eqref{choosef_n}}}{=} \lim_{\eps \to 0}\left\{\sum_{i \text{ odd}} \int_{\frac{I_i}{\eps}} W(f(x-s_i))-\lam(f^{(n-1)}(x-s_i))^2+(f^{(n)}(x-s_i))^2\dd x \right.\\
            & \qquad \quad + \! \left.\sum_{i \text{ even}} \int_{\frac{I_i}{\eps}} W(f(s_i-x))-\lam(f^{(n-1)}(s_i-x))^2+(f^{(n)}(s_i-x))^2\dd x\right\}\\
            & \overset{\eqref{choosef_n}}{\leq} C_W^{\lam,n} N + \delta.
        \end{align}
        Since $\delta \in (0,\delta_0)$ was arbitrary, the Limsup inequality follows, concluding the proof.
\end{proof}

\appendix

\section{Appendix}

For the reader's convenience, we present a suitable formulation of the Gagliardo-Nirenberg inequality \cite{Gag,Nir59} on bounded intervals, since the proof is sparsely found in the literature. Along the way, a selected range of further interpolation inequalities is shown. The Gagliardo-Nirenberg inequality is used in the proof of the Liminf inequality in Section~\ref{Subsec:Gamman}, and the nonlinear interpolation inequality in Section~\ref{Subsec:NonLin} is based on it. We build upon the proof of Fiorenza et al. \cite{FioForRosSou}, with small adjustments and simplifications which were made in consultation with the authors. As there, we start with an interpolation lemma (see \cite[Lemma~3.2]{FioForRosSou}).

\begin{rem}[Absolute continuity] \label{abscont}
    For any measurable subset $U \subset \R$ and $1 \leq p < \infty$, claiming that the $n$-th derivative $u^{(n)}$ of a function $u$ is in $L^p(U)$ implies that $u^{(n-1)}$ is in $W^{1,1}_{\mathrm{loc}}(U)$, which is equivalent to $u^{(n-1)}$ being absolutely continuous on any compact subset of $U$.
\end{rem}

\begin{lemma} \label{IntLem}
Let $I \subset \R$ be an open bounded interval and $1 \leq p,q,r < \infty$. Then there exists $C=C(q)>1$ such that for all $u \in L^q(I)$ with $u'' \in L^{r}(I)$, we have
\begin{align} \label{eq:IntLem}
    \|u'\|_{L^p(I)} \leq C\left( |I|^{1+\frac 1p - \frac 1r}\|u''\|_{L^r(I)} + |I|^{-1+\frac 1p - \frac 1q}\|u\|_{L^q(I)}\right).
\end{align}
\end{lemma}

From the proof, we find that the constant is given by $C=4 \left( \int_{(-\frac 12,\frac 12)} |x|^q\dd x \right)^{-\frac1q}=8(q+1)^{\frac 1q}$.

\begin{proof}
We restrict the proof to intervals of the form $I=(-\frac {|I|}{2}, \frac{|I|}{2})$ given that \eqref{eq:IntLem} is invariant under translations. First, let $I=(-\frac 12 , \frac 12)$. Without loss of generality, we assume $\int_I u(x) \dd x=0$. We set $\xi \coloneq  \int_I u'(x)\dd x$ and $\tilu(x) \coloneq  u(x) - g(x)$ where $g(x) \coloneq  \xi x$, and note
\begin{align}
    \int_I \tilu(y)\dd y & = -\xi\int_I y\dd y = 0, \label{tilu=0} \\
    \int_I \tilu'(y)\dd y & = \xi - \int_I \xi\dd y = 0. \label{abltilu=0}
\end{align}
From \eqref{abltilu=0} and the Fundamental Theorem of Calculus, using $\tilu''=u''$, we obtain
\begin{align} \label{abschabltilu}
\begin{split}
    \|\tilu'\|_{L^p(I)}^p & \overset{\eqref{abltilu=0}}{=}  \int_I \left| \int_I \tilu'(x) -  \tilu'(y)\dd y \right|^p \dd x
    =  \int_I \left| \int_I \int_y^x u''(z)\dd z \dd y \right|^p \dd x \\
    & \leq \int_I \left| \int_I |u''(z)| \dd z \right|^p \dd x
    = \|u''\|_{L^1(I)}^p \leq \|u''\|_{L^r(I)}^p,
\end{split}
\end{align}
and by similar arguments,
\begin{align} \label{abschtilu}
     \|\tilu\|_{L^q(I)}^q & \overset{\eqref{tilu=0}}{=}  \int_I \left| \int_I \tilu(x) -  \tilu(y)\dd y \right|^q \dd x
     \leq \|\tilu'\|_{L^1(I)}^q 
     \leq \|\tilu'\|_{L^p(I)}^q \overset{\eqref{abschabltilu}}{\leq} \|u''\|_{L^r(I)}^q.
\end{align}
Due to $u'= \tilu' +\xi$, we then have
\begin{align}
    \|u'\|_{L^p(I)} & \overset{\hphantom{\eqref{abschabltilu}}}{\leq} \|\tilu'\|_{L^p(I)} + |\xi|
    = \|\tilu'\|_{L^p(I)} + \frac{\|g\|_{L^q(I)}}{\|\text{id}\|_{L^q(I)}} \\
    & \overset{\hphantom{\eqref{abschabltilu}}}{\leq} \|\text{id}\|_{L^q(I)}^{-1} \left( \|\tilu'\|_{L^p(I)} + \|\tilu\|_{L^q(I)} + \|u\|_{L^q(I)} \right) \\
    & \overset{\eqref{abschabltilu}}{\leq} \|\text{id}\|_{L^q(I)}^{-1} \left( \|u''\|_{L^r(I)} + \|\tilu\|_{L^q(I)} + \|u\|_{L^q(I)} \right) \\
    & \overset{\eqref{abschtilu}}{\leq} 2 \|\text{id}\|_{L^q(I)}^{-1} \left( \|u''\|_{L^r(I)} + \|u\|_{L^q(I)} \right),
\end{align}
which verifies the claim for $I=(-\frac 12 , \frac 12)$. For general $I=(-\frac {|I|}{2}, \frac{|I|}{2})$ and $u \in L^q(I)$ with $u'' \in L^{r}(I)$, we define $v \coloneq  u({|I|} \cdot)$ and observe that $v \in L^q((-\frac 12 , \frac 12))$ and $v'' \in L^{r}((-\frac 12 , \frac 12))$. By applying the previously shown statement to $v$, we obtain
\begin{align} 
    \|u'\|_{L^p(I)} & = |I|^{\frac 1p -1} \| v' \|_{L^p((-\frac 12 , \frac 12))}
    \leq C |I|^{\frac 1p -1} \left( \| v'' \|_{L^r((-\frac 12 , \frac 12))} + \| v \|_{L^q((-\frac 12 , \frac 12))} \right) \\
    & = C \left(|I|^{1+ \frac 1p - \frac 1r} \|u''\|_{L^r(I)} + |I|^{ -1 + \frac 1p -\frac 1q} \|u\|_{L^q(I)} \right),
\end{align}
concluding the proof.
\end{proof}

With Lemma~\ref{IntLem}, we are now in a position to show the Gagliardo-Nirenberg inequality on the real line (see also \cite[Theorem~1.2]{FioForRosSou}).

\begin{theorem}[Gagliardo-Nirenberg inequality on the real line] \label{GagNir_R}
Let $1 \leq p,q,r < \infty$,
and $j,m \in \N_0$ where $j<m$ such that
\begin{align} \label{bedGN_R}
    \frac 1p = j + \theta \left( \frac1r - m \right) + \frac{1-\theta}{q} \quad \text{and} \quad \frac jm \leq \theta < 1.
\end{align}
Then there exists $C=C(q,r,m,j,\theta)>0$ such that for all $u \in L^q(\R)$ with $u^{(m)} \in L^r(\R)$, we have
    \begin{align} \label{eq:GagNirR}
    \|u^{(j)}\|_{L^p(\R)} \leq C \|u^{(m)}\|^{\theta}_{L^r(\R)} \|u\|^{1-\theta}_{L^q(\R)}.
\end{align}
\end{theorem}

\begin{proof}
Representatively, we show \eqref{eq:GagNirR} for the case $\theta = \frac jm$, which is also the choice of $\theta$ in the proof of Proposition~\ref{prop:liminf_inequality}. The general case follows through interpolation (see \cite{MolRosSou} for details). 
The strategy consists of first proving the statement for $j=1$ and $m=2$ using Lemma~\ref{IntLem} and a covering argument, followed by an induction on $m$ and $j$.

\textit{Step 1: Covering.} Let $u \in L^q(\R)$ and $u'' \in L^r(\R)$. In particular, this implies $u \in W^{2,1}_{\mathrm{loc}}(\R)$, and therefore we may consider the continuous representatives of $u$ and $u'$. For $k \in \N$, we set $E_k \coloneq \{x \in [-k,k]: |u(x)| \geq \frac 1k \}$, and denote by $\chi_A$ the indicator function of a set $A \subset \R$. We aim to show that there exists $k_0 \in \N$ such that for every $k > k_0$ there are finitely many open bounded intervals $I_{k,1},\ldots,I_{k,N}$ fulfilling 
\begin{enumerate}
    \item[\eni] $E_k \subset \bigcup_{i=1}^N I_{k,i}$,
    \item[\enii] $|I_{k,i}|^{-1+\frac 1p - \frac 1q} \|u\|_{L^q(I_{k,i})} = |I_{k,i}|^{1+\frac 1p - \frac 1r} \|u''\|_{L^r(I_{k,i})}$ for all $i \in \{1,\ldots,N\}$,
    \item[\eniii] $\sum_{i=1}^N \chi_{I_{k,i}} \leq 2$.
\end{enumerate}

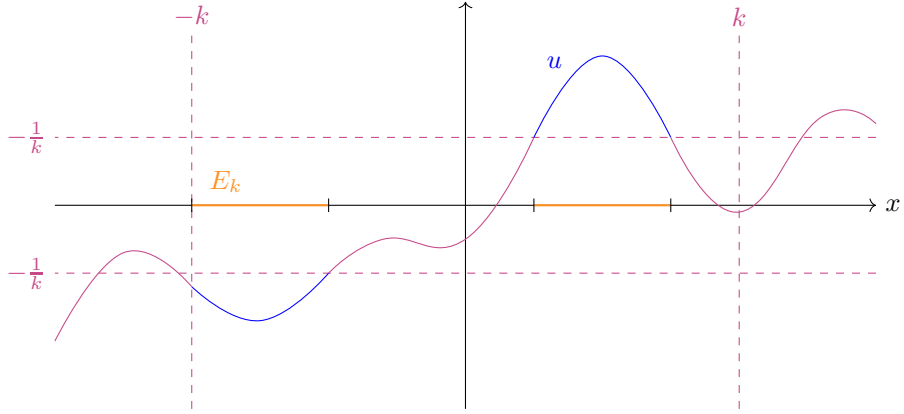
\begin{figure}[htb]
\centering
\begin{tikzpicture}[scale=0.9]

 %Achsen
   \draw[->] (-6,0) -- (6,0) node[right] {$x$};
   \draw[->] (0,-3) -- (0,3); %node[above] {\textcolor{blue}{$u(x)$}};
   \node[orange!85!white]  at (1.3,2.1) {\textcolor{blue}{$u$}};

%m und 1/m
   \draw[-,magenta!80!black,dashed] (6,-1) -- (-6,-1) node[left] {\textcolor{magenta!80!black}{$- \frac 1k$}};
   \draw[-,magenta!80!black,dashed] (6,1) -- (-6,1) node[left] {\textcolor{magenta!80!black}{$- \frac 1k$}};;
  \draw[-,magenta!80!black,dashed] (-4,-3) -- (-4,2.5)node[above] {\textcolor{magenta!80!black}{$-k$}};
   \draw[-,magenta!80!black,dashed] (4,-3) -- (4,2.5)node[above] {\textcolor{magenta!80!black}{$k$}};

%Funktion
\draw[magenta!80!black]  plot[smooth, tension=0.8] coordinates {(-6,-2) (-5,-0.7) (-4,-1.2)};
\draw[blue]  plot[smooth, tension=0.8] coordinates {(-4,-1.2) (-3,-1.7) (-2,-1)};
\draw[magenta!80!black]  plot[smooth, tension=0.8] coordinates {(-2,-1) (-1.2,-0.5) (0,-0.5) (1,1)};
\draw[blue]  plot[smooth, tension=0.8] coordinates {(1,1) (2,2.2) (3,1)};
\draw[magenta!80!black]  plot[smooth, tension=0.8] coordinates {(3,1) (4,-0.1) (5.2,1.3) (6,1.2)};

%E_m
 \draw[-,orange!85!white,thick] (-4,0) -- (-2,0);
 \draw[-,orange!85!white,thick] (1,0) -- (3,0);
  \draw (-4,-.1) -- (-4,.1);
 \draw (-2,-.1) -- (-2,.1);
 \draw (1,-.1) -- (1,.1);
 \draw (3,-.1) -- (3,.1);
 \node[orange!85!white]  at (-3.5,0.35) {$E_k$};

\end{tikzpicture} 
\caption{The aim is to cover the set $E_k$ by open bounded intervals. This enables the application of Lemma~\ref{IntLem}.}
\end{figure}

\noindent Without loss of generality we assume $u \not\equiv 0$. For fixed $k \in \N$ such that $E_k$ is non-empty we fix $x \in E_k$, and define $\phi_x,\psi_x : (0, \infty) \to \R$ by
\begin{align} \label{defphipsi}
    \phi_x(h)\coloneq h^{-1+\frac 1p - \frac 1q} \|u\|_{L^q((x-\frac h2, x+\frac h2))} \quad \text{and} \quad 
    \psi_x(h)\coloneq h^{1+\frac 1p - \frac 1r} \|u''\|_{L^r((x-\frac h2, x+\frac h2))}.
\end{align}
We note that both functions are strictly positive. Since $u$ is absolutely continuous on $E_k$, for every $\eps \in (0, \frac 12)$ there exists $h > 0$ such that for all $t \in (x- \frac h2, x + \frac h2)$ we have
\begin{align}
    |u(x)-u(t)| < \frac{\eps}{2k}.
\end{align}
Furthermore, since $|u(x)| \geq \frac 1k$, we can choose $h>0$ sufficiently small to ensure $|u(t)| \geq \frac{1}{2k}$ for all $t \in (x- \frac h2, x + \frac h2)$, and thus
\begin{align} \label{hklein1}
    |u(x)-u(t)| < \eps |u(t)|.
\end{align}
Additionally, we choose $h$ to be sufficiently small such that
\begin{align} \label{hklein2}
    \|u''\|_{L^r(\R)} < h^{-2+\frac 1r} \frac{1}{3k} \leq h^{-2+\frac 1r} \frac{|u(x)|}{3}.
\end{align}
Hence, for such an $h > 0$ we have
\begin{align}
    \psi_x(h) & \overset{\hphantom{\eqref{hklein1}}}{\leq} h^{1+ \frac 1p - \frac 1r} \|u''\|_{L^r(\R)} 
    \overset{\eqref{hklein2}}{\leq} h^{-1 + \frac 1p} \frac{|u(x)|}{3} = h^{\frac 1p - \frac 1q -1} \left( \int_{x-\frac h2}^{x+ \frac h2} \left(\frac{|u(x)|}{3}\right)^q\dd t\right)^{\frac 1q} \\
    & \overset{\eqref{hklein1}}{\leq} h^{\frac 1p - \frac 1q -1} \left( \int_{x-\frac h2}^{x+ \frac h2} \left(\frac{|u(t)|(1+\eps)}{3}\right)^q\dd t\right)^{\frac 1q} 
    \! \leq h^{\frac 1p - \frac 1q -1} \left( \int_{x-\frac h2}^{x+ \frac h2} \left(\frac{|u(t)|}{2}\right)^q\dd t\right)^{\frac 1q} 
    \! \leq \phi_x(h).
\end{align}
Accordingly, it follows that the set ${\{h>0: \psi_x(h) \leq \phi_x(h)\}}$ is non-empty. The set is also bounded, as there exists $h_x > 0$ such that $\psi_x(h_x)=\phi_x(h_x)$ and $\psi_x(h) > \phi_x(h)$ for $h>h_x$. The latter is evident from the examination of the exponents in \eqref{defphipsi}. More precisely, since $\frac 1p - \frac 1q -1< 0 < 1 + \frac 1p - \frac 1r$ and $u \in L^q(\R)$, the expression $\phi_x(h)$ is bounded, whereas $\psi_x(h)$ approaches infinity as $h \to \infty$. Now, consider the open covering $\bigcup_{x \in E_k} \left(x-\frac {h_x}2, x+\frac {h_x}2\right)$ of $E_k$. Since $E_k$ is compact, we can choose a finite subcovering 
\begin{align}
    \bigcup_{i=1}^{\tilde{N}} I_{k,i} \quad \text{where} \quad I_{k,i} \coloneq \left(x_i-\frac {h_i}2,x_i + \frac{h_i}2\right).
\end{align} 
The previous considerations ensure the validity of \enii. Lastly, we select a subsystem such that \eniii\ is also satisfied. If there exists $x \in \bigcap_{j=1}^M I_{k,i_j}$ with $M \geq 3$ and $i_j \in \{1,\ldots,\tilde{N}\}$ pairwise different, we define
\begin{align}
    a_x \coloneq  \min_{1 \leq i \leq \tilde{N}} \{\inf(I_{k,i}): x \in I_{k,i}\} \quad \text{and} \quad b_x\coloneq  \max_{1 \leq i \leq \tilde{N}}\{\sup(I_{k,i}): x \in I_{k,i}\}
\end{align}
and omit all intervals $I_{k,i_j}$ from the covering except for one of the form $(a_x,s)$ with $s \in \R$ and one of the form $(t,b_x)$ with $t \in \R$. By repeating this process for all $x$ that lie in more than two intervals of the covering, we obtain, after a finite number of steps, the desired covering $I_{k,1},\ldots,I_{k,N}$ satisfying \eni--\eniii.

\textit{Step 2: Extension to $\R$.} Now, we are in a position to show \eqref{eq:GagNirR} for $j=1$ and $m=2$ using Lemma~\ref{IntLem}. In this case, condition \eqref{bedGN_R} translates to
\begin{align} \label{bed12GN}
   \frac 2p = \frac 1r + \frac 1q
\end{align}
for $1 \leq p,q,r < \infty$, and we have $\theta=1-\theta=\frac12$. Exploiting the properties of the covering from \textit{Step 1} and applying Lemma~\ref{IntLem}, which is applicable to open bounded intervals, we get
\begin{align}
  \|u'\|_{L^p(E_k)}^p & \overset{\ \,\eni\,\ }{\leq} \sum_{i=1}^N \|u'\|_{L^p(I_{k,i})}^p 
  \leq C \sum_{i=1}^N \left( |I_{k,i}|^{p+1 - \frac pr}\|u''\|_{L^r(I_{k,i})}^p + |I_{k,i}|^{-p+ 1 - \frac pq}\|u\|_{L^q(I_{k,i})}^p\right) \\
  & \overset{\ \enii\ }{=} 2 C \sum_{i=1}^N |I_{k,i}|^{p+1 - \frac pr}\|u''\|_{L^r(I_{k,i})}^{\frac p2} \left( |I_{k,i}|^{-2 -\frac 1q + \frac 1r}\|u\|_{L^q(I_{k,i})} \right)^{\frac p2} \\
  & \overset{\eqref{bed12GN}}{=} 2 C \sum_{i=1}^N \left( \|u''\|_{L^r(I_{k,i})}^r \right)^{\frac{q}{q+r}} \left( \|u\|_{L^q(I_{k,i})}^q \right)^{\frac{r}{q+r}}.
\end{align}
With that, using Hölder's inequality for sums, we further estimate
\begin{align}
  \|u'\|_{L^p(E_k)}^p
  & \overset{\hphantom{\eniii}}{\leq} 2 C \left( \sum_{i=1}^N \|u''\|_{L^r(I_{k,i})}^r \right)^{\frac{q}{q+r}} \left(\sum_{i=1}^N \|u\|_{L^q(I_{k,i})}^q \right)^{\frac{r}{q+r}} \\
  & \overset{\eniii}{\leq} 4 C \|u''\|_{L^r(\R)}^{\frac{qr}{q+r}} \|u\|_{L^q(\R)}^{\frac{qr}{q+r}}
  \overset{\eqref{bed12GN}}{=} 4 C \|u''\|_{L^r(\R)}^{\frac p2} \|u\|_{L^q(\R)}^{\frac p2}.
\end{align}
 Due to the continuity of $u'$, the set $\{x \in \R: u'(x) \neq 0\}$ is open and thus can be represented as a disjoint union of countably many open intervals $\{J_k\}_{k \in \N}$ (see \cite[Chapter~1, Proposition~9]{Roy}). We note that $u$ is strictly monotone on each of these intervals. Consequently, we have that $J_k \cap \{x \in \R: u(x)=0\}$ consists of at most one point for every $k \in \N$ and conclude that 
\begin{align}
    \{x \in \R: u'(x)\neq0 \text{ and } u(x)=0 \} = \bigcup_{k \in \N} (J_k \cap \{x \in \R: u(x)=0\})
\end{align} 
is countable, i.e. has Lebesgue measure zero. Accordingly, $u'=u' \chi_{\{u\neq0\}}$ holds almost everywhere in $\R$, and since $|u' \chi_{E_k}| \nearrow |u' \chi_{\{u\neq0\}}|$ as $k \to \infty$, the Monotone Convergence Theorem yields $\|u'\|_{L^p(E_k)} \nearrow \|u'\|_{L^p(\R)}$. In conclusion, we have
\begin{align} \label{eq:GagNir12}
    \|u'\|^2_{L^p(\R)} = \sup_{k \in \N} \|u'\|^2_{L^p(E_k)} \leq C \|u''\|_{L^r(\R)} \|u\|_{L^q(\R)},
\end{align}
which corresponds to \eqref{eq:GagNirR} for $j=1$ and $m=2$.

\textit{Step 3: Induction on $m$ and $j$.} 
Finally, we prove \eqref{eq:GagNirR} for $\theta=\frac jm$ by induction on the orders. Initially, we consider $j=1$ and perform an induction on $m$ where the case $m=2$ is covered by \textit{Step 2}. We assume the statement holds for $j=1$ and a fixed $m \in \N_{\geq2}$. Now, let \eqref{bedGN_R} be fulfilled for $m+1$, i.e.
\begin{align} \label{bedm+1GN}
    \frac 1p = \frac 1{(m+1)r} + \frac m{(m+1)q}.
\end{align}
Additionally, we define the exponent $\tilp\coloneq  \frac{pq}{2q-p}$ such that
\begin{align} \label{bedtilp}
    \frac 2p = \frac 1\tilp + \frac 1q.
\end{align}
Comparing this with \eqref{bed12GN}, we notice that we can apply \eqref{eq:GagNir12} and subsequently use the induction hypothesis for $u'$, such that
\begin{align}
    \|u'\|^2_{L^p(\R)} & \leq C \|u''\|_{L^{\tilp}(\R)} \|u\|_{L^q(\R)} \leq C \|u^{(m+1)}\|_{L^r(\R)}^{\frac 1m} \|u'\|_{L^p(\R)}^{1-\frac 1m} \|u\|_{L^q(\R)}.
\end{align}
Here, the induction hypothesis is applicable since
\begin{align}
    \frac 1\tilp \overset{\eqref{bedtilp}}{=} \frac 2p - \frac 1q 
    \overset{\eqref{bedm+1GN}}{=} \frac{m-1}{mp} + \frac 1{mr}.
\end{align}
By rearranging, we obtain
\begin{align}
    \|u'\|_{L^p(\R)} \leq C \|u^{(m+1)}\|_{L^r(\R)}^{\frac 1{m+1}} \|u\|_{L^q(\R)}^{\frac m{m+1}},
\end{align}
which corresponds to the statement for $m+1$. 

Now, we perform another induction to prove \eqref{eq:GagNirR}. Initially, we suppose the statement holds for fixed arbitrary $j,m \in \N$ with $j<m$ and aim to deduce the statement for $j+1$ and $m+1$. Accordingly, we assume
\begin{align} \label{bedj+1GN}
    \frac 1p = \frac{j+1}{(m+1)r} + \frac{m-j}{(m+1)q},
\end{align}
and define the exponent $\tilq\coloneq  \frac{pr(m-j)}{mr-jp}$ such that $\frac 1p = \frac j{mr} + \frac{m-j}{m\tilq}$. Applying the induction hypothesis to $u'$ yields
\begin{align}
    \|u^{(j+1)}\|_{L^p(\R)} \leq C \|u^{(m+1)}\|_{L^r(\R)}^{\frac jm} \|u'\|_{L^{\tilq}(\R)}^{1-\frac jm}.
\end{align}
From the previous induction we also obtain
\begin{align}
    \|u'\|_{L^{\tilq}(\R)} \leq C \|u^{(j+1)}\|_{L^p(\R)}^{\frac1{j+1}}\|u\|_{L^q(\R)}^{1-\frac 1{j+1}}
\end{align}
which holds due to
\begin{align}
    \frac 1\tilq & = \frac1{m-j} \left(\frac mp - \frac jr\right) 
    \overset{\eqref{bedj+1GN}}{=}\frac1{m-j} \left(\frac mp - j\frac {q(m+1)-p(m-j)}{pq(j+1)}\right)
    =\frac 1{(j+1)p} + \frac j{(j+1)q}.
\end{align}
Together, we conclude
\begin{align}
    \|u^{(j+1)}\|_{L^p(\R)} \leq C \|u^{(m+1)}\|_{L^r(\R)}^{\frac {j+1}{m+1}} \|u\|_{L^q(\R)}^{\frac{m-j}{m+1}}.
\end{align}
This completes the induction, and we obtain the validity of \eqref{eq:GagNirR} under condition \eqref{bedGN_R} for arbitrary $j,m \in \N$ with $j<m$ and $\theta=\frac jm$.
\end{proof}

 From Theorem~\ref{GagNir_R}, we derive the corresponding statement on bounded intervals using an extension operator introduced by Rogers \cite{Rog}. Before that, we establish another auxiliary interpolation inequality, noting that an even stronger result is proved in \cite{LiZha}.

\begin{lemma} \label{IntLemAbstr}
Let $I \subset \R$ be an open bounded interval, $1 \leq q,r < \infty$ and $j,m \in \N_0$ where $j<m$. Then, there exists $C=C(q,r,m,I)>0$ such that for all $u \in L^q(I) \cap W^{m,r}(I)$, we have
\begin{align} \label{eq:IntLemAbstr}
    \|u^{(j)}\|_{L^r(I)} \leq \|u^{(m)}\|_{L^r(I)} + C\|u\|_{L^q(I)}.
\end{align}
\end{lemma}

\begin{proof}
We argue by contradiction and suppose there exists a sequence $u_n \in W^{m,r}(I)$ fulfilling
    \begin{align} \label{assume}
        \|u_n^{(m)}\|_{L^r(I)} + n\|u_n\|_{L^q(I)} \leq \|u_n^{(j)}\|_{L^r(I)}.
    \end{align}
We define the sequence $v_n \coloneq u_n \left(\|u_n\|_{W^{m-1,r}(I)}+\|u_n^{(j)}\|_{L^r(I)}\right)^{-1} \in W^{m,r}(I)$ such that
\begin{align} \label{constseq}
    \|v_n\|_{W^{m-1,r}(I)}+\|v_n^{(j)}\|_{L^r(I)}=1.
\end{align}
Since \eqref{assume} also holds for the sequence $v_n$, we then obtain
\begin{align} \label{bounded1}
    \|v_n^{(m)}\|_{L^r(I)}+n\|v_n\|_{L^q(I)} 
    \leq \|v_n^{(j)}\|_{L^r(I)}
    \overset{\eqref{constseq}}{\leq} 1-\|v_n\|_{W^{m-1,r}(I)} \leq 1,
\end{align}
and thereby
\begin{align} \label{bounded2}
\begin{split}
    \|v_n\|_{W^{m,r}(I)} & \overset{\eqref{bounded1}}{\leq}  \|v_n\|_{W^{m-1,r}(I)} +1 - n \|v_n\|_{L^q(I)}\\
    &\overset{\eqref{constseq}}{\leq} 2- \|v_n^{(j)}\|_{L^r(I)} - n \|v_n\|_{L^q(I)}
    \leq 2.
\end{split}
\end{align}
Hence, there exist $\tilv \in W^{m,r}(I)$ and a subsequence of $v_n$ (not relabeled), such that $v_n^{(i)} \rightharpoonup \tilv^{(i)}$ in $L^r(I)$ as $n \to \infty$ for all $i \in \{0,\ldots , m\}$. Since $W^{1,r}(I)$ is compactly embedded into $L^r(I)$, we obtain strong convergence of a further subsequence for the derivatives up to order $m-1$, that is
\begin{align}
    \|v_n-\tilv \|_{W^{m-1,r}(I)}+ \|v_n^{(j)}- \tilv^{(j)}\|_{L^r(I)} = \sum_{i=0}^{m-1} \|v_n^{(i)}- \tilv^{(i)}\|_{L^r(I)} +\|v_n^{(j)}- \tilv^{(j)}\|_{L^r(I)} \to 0
\end{align}
as $n \to \infty$. At the same time, we have 
\begin{align}
    \|v_n-\tilv \|_{W^{m-1,r}(I)}+ \|v_n^{(j)}- \tilv^{(j)}\|_{L^r(I)} 
    \overset{\eqref{constseq}}{\geq} 1 - \|\tilv \|_{W^{m-1,r}(I)} - \|\tilv^{(j)}\|_{L^r(I)},
\end{align}
implying $\|\tilv \|_{W^{m-1,r}(I)} + \|\tilv^{(j)}\|_{L^r(I)} \geq 1$.
From \eqref{bounded1}, we further deduce
\begin{align}
    \lim_{n\to \infty} \|v_n\|_{L^q(I)} \leq \lim_{n \to \infty} \frac 1n \left(1-\|v_n^{(m)}\|_{L^r(I)}\right) \overset{\eqref{bounded2}}{=} 0.
\end{align}
Thus, $v_n$ has a subsequence that converges to zero almost everywhere. On the other hand, we find a further subsequence that converges almost everywhere to $\tilv$ due to the strong convergence in $L^r(I)$. Consequently, we have $\tilv = 0$ almost everywhere, which is a contradiction.
\end{proof}

From Theorem~\ref{GagNir_R} and Lemma~\ref{IntLemAbstr} we can now derive the Gagliardo-Nirenberg inequality on bounded intervals using a suitable extension operator for Sobolev spaces from intervals to the real line (see \cite{Rog}). An analogous approach for general bounded Lipschitz domains in $\R^n$ is described in \cite[Theorem~1.3]{LiZha}.

\begin{theorem}[Gagliardo-Nirenberg inequality on bounded intervals] \label{GagNir}
Let $I \subset \R$ be an open bounded interval, $1 \leq p,q,r < \infty$,
and $j,m \in \N_0$ where $j<m$ such that
\begin{align} \label{bedGN}
    \frac 1p = j + \theta \left( \frac1r - m \right) + \frac{1-\theta}{q} \quad \text{and} \quad \frac jm \leq \theta < 1.
\end{align}
Then there exists $C=C(q,r,m,j,\theta,I)>0$ such that for all $u \in L^q(I) \cap W^{m,r}(I)$, we have
\begin{align} \label{eq:GagNir}
    \|u^{(j)}\|_{L^p(I)} \leq C\left( \|u^{(m)}\|^{\theta}_{L^r(I)} \|u\|^{1-\theta}_{L^q(I)} + \|u\|_{L^q(I)}\right).
\end{align}
\end{theorem}

\begin{proof}
    In \cite[Theorem~8]{Rog}, it is shown that for any $1 \leq s < \infty$ and $k \in \N_0$ there exists a bounded linear extension operator $T \colon W^{k,s}(I) \to W^{k,s}(\R)$ and a constant $c=c(s,k,I)>0$ such that for all $u \in W^{k,s}(I)$ we have
    \begin{align} \label{extop}
        \|Tu\|_{W^{k,s}(\R)} \leq c \|u\|_{W^{k,s}(I)}.
    \end{align}
    Thus, with Lemma~\ref{IntLemAbstr} and Theorem~\ref{GagNir_R}, for $u \in L^q(I) \cap W^{m,r}(I)$ we obtain
    \setlength\jot{0.6cm}
    \begin{align}
        \|u^{(j)}\|_{L^p(I)} & \overset{\hphantom{\eqref{extop}}}{\leq} \|(T u)^{(j)}\|_{L^p(\R)}
        \overset{\eqref{eq:GagNirR}}{\leq} C \|Tu\|^{\theta}_{W^{m,r}(\R)} \|Tu\|^{1-\theta}_{L^q(\R)}\\
        &\overset{\eqref{extop}}{\leq} C \Big(\sum_{i=0}^{m-1} \|u^{(i)}\|_{L^{r}(I)} + \|u^{(m)}\|_{L^{r}(I)}\Big)^{\theta} \|u\|^{1-\theta}_{L^q(I)}\\
        & \overset{\eqref{eq:IntLemAbstr}}{\leq} C \Big(\|u\|_{L^q(I)}+ \|u^{(m)}\|_{L^{r}(I)}\Big)^{\theta} \|u\|^{1-\theta}_{L^q(I)}\\
        & \overset{\hphantom{\eqref{extop}}}{\leq} C \Big(\|u^{(m)}\|^\theta_{L^{r}(I)} \|u\|^{1-\theta}_{L^q(I)} + \|u\|_{L^q(I)}\Big),
    \end{align}
    where the last step holds due to the subadditivity of $x \mapsto x^\theta$ for $x \in [0, \infty)$ and $\theta \in (0,1)$.
\end{proof}

\section*{Acknowledgments}
We like to thank Lucas Fix, Lucas Schmitt, Jan-Frederik Pietschmann, and Caterina Zeppieri for highly valuable discussions and advices on the topic. Part of this work was carried out while DB and GG were affiliated with Heidelberg University. HK is funded by the Deutsche Forschungsgemeinschaft (DFG, German Research Foundation) under Germany’s Excellence Strategy EXC-2181/1-39090098 (the Heidelberg STRUCTURES Cluster of Excellence).

\bibliographystyle{alphaabbr}
\bibliography{bib}
\end{document}